\documentclass[10pt,a4paper]{article}
\usepackage{amsmath,amsfonts,amsthm}
\usepackage{amssymb}
\usepackage{authblk}
\usepackage{booktabs}
\usepackage{cancel}
\usepackage{caption}
\usepackage{comment}
\usepackage{pgfplots}
\usepackage[hmargin={28mm,28mm},vmargin={30mm,35mm}]{geometry}
\usepackage{hyperref}
\usepackage{mathabx}
\usepackage{paralist}

\usepackage{multirow}

\usepackage{caption}
\usepackage{subcaption}

\usepackage{helvet}
\usepackage{courier}

\usepackage[bold]{hhtensor}

\newcommand{\email}[1]{\href{mailto:#1}{#1}}

\usepackage[parfill]{parskip}
\begingroup
    \makeatletter
    \@for\theoremstyle:=definition,remark,plain\do{%
        \expandafter\g@addto@macro\csname th@\theoremstyle\endcsname{%
            \addtolength\thm@preskip\parskip
            }%
        }
\endgroup

\theoremstyle{theorem}
\newtheorem{theorem}{Theorem}
\newtheorem{lemma}[theorem]{Lemma}
\newtheorem{proposition}[theorem]{Proposition}

\theoremstyle{remark}
\newtheorem{remark}[theorem]{Remark}
\theoremstyle{definition}

\newcommand{\Real}{\mathbb{R}}

\newcommand{\Poly}[1]{\mathbb{P}^{#1}}

\newcommand{\SCAL}{{\cdot}}
\newcommand{\GRAD}{\nabla}
\newcommand{\DIV}{\nabla{\cdot}}
\newcommand{\LAPL}{{\Delta}}

\newcommand{\norm}[2][]{\|#2\|_{#1}}
\newcommand{\seminorm}[2][]{|#2|_{#1}}
\newcommand{\tnorm}[2][]{\vvvert #2\vvvert_{#1}}

\newcommand{\term}{\mathfrak{T}}

\newcommand{\Mh}[1][h]{\mathcal{M}_{#1}}
\newcommand{\Th}[1][h]{\mathcal{T}_{#1}}
\newcommand{\Fh}[1][h]{\mathcal{F}_{#1}}
\newcommand{\Fhi}{\Fh^{{\rm i}}}
\newcommand{\Fhb}{\Fh^{{\rm b}}}

\newcommand{\normal}{\vec{n}}

\newcommand{\ST}{\; : \;}

\newcommand{\visc}{\nu}
\newcommand{\vel}{\vec{\beta}}
\newcommand{\svel}[1][TF]{\beta_{#1}}
\newcommand{\reac}{\mu}

\newcommand{\UT}[1][k]{\underline{\vec{U}}_T^{#1}}
\newcommand{\Uh}[1][k]{\underline{\vec{U}}_h^{#1}}
\newcommand{\UhD}[1][k]{\underline{\vec{U}}_{h,0}^{#1}}

\newcommand{\Ph}[1][k]{P_h^{#1}}

\newcommand{\DT}[1][k]{D_T^{#1}}

\newcommand{\rT}[1][k+1]{\vec{r}_T^{#1}}
\newcommand{\Gb}[1][k]{\vec{G}_{\vel,T}^{#1}}

\newcommand{\dT}[1][k]{\vec{\delta}_{T}^{#1}}
\newcommand{\dTF}[1][k]{\vec{\delta}_{TF}^{#1}}

\newcommand{\IT}[1][k]{\underline{\vec{I}}_T^{#1}}
\newcommand{\Ih}[1][k]{\underline{\vec{I}}_h^{#1}}

\newcommand{\vU}{\vec{U}}
\newcommand{\vx}[1][]{\vec{x}}
\newcommand{\vu}[1][]{\vec{u}_{#1}}
\newcommand{\vv}[1][]{\vec{v}_{#1}}
\newcommand{\vw}[1][]{\vec{w}_{#1}}

\newcommand{\cvw}[1][T]{\check{\vec{w}}_{#1}}

\newcommand{\uvu}[1][h]{\underline{\vec{u}}_{#1}}
\newcommand{\uvv}[1][h]{\underline{\vec{v}}_{#1}}
\newcommand{\uvw}[1][h]{\underline{\vec{w}}_{#1}}

\newcommand{\ve}[1][]{\vec{e}_{#1}}
\newcommand{\uve}[1][h]{\underline{\vec{e}}_{#1}}

\newcommand{\vf}[1][]{\vec{f}_{#1}}

\newcommand{\vhw}[1][T]{\widehat{\vec{w}}_{#1}}
\newcommand{\uvhu}[1][h]{\underline{\widehat{\vec{u}}}_{#1}}
\newcommand{\hp}[1][h]{\widehat{p}_h}

\newcommand{\lproj}[2][h]{\pi_{#1}^{0,#2}}
\newcommand{\eproj}[2][h]{\pi_{#1}^{1,#2}}
\newcommand{\vlproj}[2][h]{\boldsymbol{\pi}_{#1}^{0,#2}}
\newcommand{\veproj}[2][T]{\boldsymbol{\pi}_{#1}^{1,#2}}

\newcommand{\Pe}{\mathrm{Pe}}

\newcommand{\Eh}[1][h]{\mathcal{E}_{#1}}

\newcommand{\cN}[1][]{\mathcal{N}_{#1}}

\newcommand{\tref}[1][T]{\tau_{{\rm ref},#1}}
\newcommand{\vref}[1][T]{\beta_{{\rm ref},#1}}

\newcommand{\Ca}{C_{\rm a}}
\newcommand{\Cb}{C_{\rm b}}

\newcommand{\uDpT}[1][k]{\underline{\vec{\Delta}}_{\partial T}^{#1}}

\graphicspath{{figures/pdf/}}

\newcommand{\logLogSlopeTriangle}[5]
{

    \pgfplotsextra
    {
        \pgfkeysgetvalue{/pgfplots/xmin}{\xmin}
        \pgfkeysgetvalue{/pgfplots/xmax}{\xmax}
        \pgfkeysgetvalue{/pgfplots/ymin}{\ymin}
        \pgfkeysgetvalue{/pgfplots/ymax}{\ymax}

        \pgfmathsetmacro{\xArel}{#1}
        \pgfmathsetmacro{\yArel}{#3}
        \pgfmathsetmacro{\xBrel}{#1-#2}
        \pgfmathsetmacro{\yBrel}{\yArel}
        \pgfmathsetmacro{\xCrel}{\xArel}

        \pgfmathsetmacro{\lnxB}{\xmin*(1-(#1-#2))+\xmax*(#1-#2)} 
        \pgfmathsetmacro{\lnxA}{\xmin*(1-#1)+\xmax*#1} 
        \pgfmathsetmacro{\lnyA}{\ymin*(1-#3)+\ymax*#3} 
        \pgfmathsetmacro{\lnyC}{\lnyA+#4*(\lnxA-\lnxB)}
        \pgfmathsetmacro{\yCrel}{\lnyC-\ymin)/(\ymax-\ymin)} 

        \coordinate (A) at (rel axis cs:\xArel,\yArel);
        \coordinate (B) at (rel axis cs:\xBrel,\yBrel);
        \coordinate (C) at (rel axis cs:\xCrel,\yCrel);

        \draw[#5]   (A)-- node[pos=0.5,anchor=north] {\scriptsize{1}}
                    (B)-- 
                    (C)-- node[pos=0.,anchor=west] {\scriptsize{#4}} 
                    cycle;
    }
}

\begin{document}

\title{An advection-robust Hybrid High-Order method for the Oseen problem\footnote{The work of the second author was supported by \emph{Agence Nationale de la Recherche} grant HHOMM (ANR-15-CE40-0005).}}

\author[1]{Joubine Aghili\footnote{\email{jaghili@unice.fr}}}
\author[2]{Daniele A. Di Pietro\footnote{\email{daniele.di-pietro@umontpellier.fr}}}
\affil[1]{INRIA, Universit\'{e} C\^{o}te d'Azur, Laboratoire Jean Dieudonn\'{e}, CNRS, France}
\affil[2]{Institut Montpelli\'erain Alexander Grothendieck, Univ. Montpellier, CNRS, France}

\maketitle

\begin{abstract}
  In this work, we study advection-robust Hybrid High-Order discretizations of the Oseen equations.
  For a given integer $k\ge 0$, the discrete velocity unknowns are vector-valued polynomials of total degree $\le k$ on mesh elements and faces, while the pressure unknowns are discontinuous polynomials of total degree $\le k$ on the mesh.
  From the discrete unknowns, three relevant quantities are reconstructed inside each element: a velocity of total degree $\le(k+1)$, a discrete advective derivative, and a discrete divergence.
  These reconstructions are used to formulate the discretizations of the viscous, advective, and velocity-pressure coupling terms, respectively.
  Well-posedness is ensured through appropriate high-order stabilization terms.
  We prove energy error estimates that are advection-robust for the velocity, and show that each mesh element $T$ of diameter $h_T$ contributes to the discretization error with an $\mathcal{O}(h_T^{k+1})$-term in the diffusion-dominated regime, an $\mathcal{O}(h_T^{k+\frac12})$-term in the advection-dominated regime, and scales with intermediate powers of $h_T$ in between.
  Numerical results complete the exposition.
\end{abstract}

\section{Introduction}

Since the pioneering works \cite{Cockburn.Shu:91,Cockburn.Shu:89,Cockburn.Lin.ea:89,Cockburn.Hou.ea:90,Cockburn.Shu:98} of Cockburn, Shu and coworkers in the late 1980s, discontinuous Galerkin (DG) methods have gained significant popularity in computational fluid mechanics, boosted by the 1997 landmark papers \cite{Bassi.Rebay:97,Bassi.Rebay.ea:97} by Bassi, Rebay and coworkers on the treatment of viscous terms.
At the roots of this success are, in particular: the possibility to handle general meshes (including, e.g., nonconforming interfaces) and high approximation orders; the robustness with respect to dominant advection; the satisfaction of local balances on the computational mesh.
The application of DG methods to the discretization of incompressible flow problems has been considered in several works starting from the early 2000s; a bibliographic sample includes \cite{Karakashian.Katsaounis:00,Cockburn.Kanschat.ea:02,Toselli:02,Cockburn.Kanschat.ea:05,Girault.Riviere.ea:05,Bassi.Crivellini.ea:06,Bassi.Crivellini.ea:07,Cockburn.Kanschat.ea:07,Di-Pietro:07,Mozolevski.Suli.ea:07,Hansbo.Larson:08,Burman.Stamm:10,Di-Pietro.Ern:10,Bassi.Botti.ea:12*1,Becker.Capatina.ea:12,Wihler.Wirz:12,Crivellini.DAlessandro.ea:13,Badia.Codina.ea:14,Riviere.Sardar:14,Tavelli.Dumbser:14}; cf. also \cite[Chapter 6]{Di-Pietro.Ern:12} for a pedagogical introduction.

Despite their numerous favorable features in the context of incompressible fluid mechanics, DG methods suffer from two major drawbacks, particularly when non-standard meshes are considered:
first, the number of unknowns rapidly grows as $k^dN_{\cal T}$, with $d$ denoting the space dimension and $N_{\cal T}$ the number of mesh elements;
second, the inf-sup condition may not be verified, which requires to add pressure stabilization terms.
Of course, these inconveniences can be overcome in specific cases; see, e.g., \cite[Section 6.1.5]{Di-Pietro.Ern:12} for a discussion on DG discretizations without pressure stabilization.

A significant contribution to the resolution of both issues was given in the 2009 paper by Cockburn, Gopalakrishnan, and Lazarov \cite{Cockburn.Gopalakrishnan.ea:09}, where the authors provide a unified framework for the hybridization of Finite Element (FE) methods for second order elliptic problems.
An important side result of this work is that discontinuous methods featuring a reduced number of globally coupled unknowns can be devised by enforcing flux continuity through Lagrange multipliers on the mesh skeleton, leading to the so-called Hybridizable Discontinuous Galerkin (HDG) methods.
A first notable consequence of introducing skeletal unknowns is that efficient implementations exploiting hybridization and static condensation are possible.
This leads to globally coupled problems where the number of unknowns grows as $k^{d-1} N_{\cal F}$, with $N_{\cal F}$ denoting the number of faces.
A second important consequence is that a Fortin interpolator is available, paving the way to inf-sup stable methods for incompressible problems on general meshes.

In their original formulation, HDG methods focused on meshes with standard element shapes, and had in some cases lower orders of convergence than standard mixed FE methods; see Remark \ref{rem:HDG.stab} on this subject.
Recently, a novel class of methods that overcome both limitations have been proposed in \cite{Di-Pietro.Ern.ea:14,Di-Pietro.Ern:15} under the name of Hybrid High-Order (HHO) methods.
The relation between HDG and HHO methods has been explored in a recent work \cite{Cockburn.Di-Pietro.ea:16}, which points out the analogies and differences among the two frameworks.
In particular, HHO-related advances include the applicability to general polyhedral meshes in arbitrary space dimension, as well as the identification of high-order stabilizing contributions which allow to gain up to one order of convergence with respect to classical HDG methods.
Additionally, powerful discrete functional analysis tools have been developed in the HHO framework that make the convergence analysis using compactness techniques possible for problems involving complex nonlinearities; see, e.g., \cite{Di-Pietro.Droniou:17,Di-Pietro.Droniou:17*1,Botti.Di-Pietro.ea:17}.
It has to be noted that also other recently developed methods support polygonal meshes and high-orders.
We cite here, in particular, the Virtual Element methods (VEM) in their primal conforming \cite{Beirao-da-Veiga.Brezzi.ea:13}, mixed \cite{Brezzi.Falk.ea:14}, and primal nonconforming \cite{Ayuso-de-Dios.Lipnikov.ea:16} flavors, as well as the recently introduced $M$-decompositions \cite{Cockburn.Fu:17,Cockburn.Fu:17*1}.
For a study of the relations among HDG, HHO, and VEM, we refer the reader to \cite{Boffi.Di-Pietro:17,Di-Pietro.Droniou.ea:18}.

HDG, HHO and related methods have been applied to the discretization of incompressible flows; see, e.g., \cite{Nguyen.Peraire.ea:11,Cesmelioglu.Cockburn.ea:13,Giorgiani.Fernandez-Mendez.ea:14,Qiu.Shi:16,Ueckermann.Lermusiaux:16,Aghili.Boyaval.ea:15,Di-Pietro.Krell:17,Beirao-da-Veiga.Lovadina.ea:17,Cesmelioglu.Cockburn.ea:17}.
In this work, we propose a novel and original study of HHO discretizations of the Oseen problem highlighting the dependence of the error estimates on the P\'eclet number when the latter takes values in $[0,+\infty)$.
Notice that $+\infty$ is excluded since we assume nonzero viscosity; we refer the reader to
\cite{Di-Pietro.Ern.ea:08,Di-Pietro.Droniou.ea:15} for the study of DG and HHO
methods for locally vanishing diffusion with advection, where this assumption
is removed.
It is also worth mentioning that this type of analysis does not seem straightforward for the Navier--Stokes problem, since error estimates typically require small data assumptions, essentially limiting the range of Reynolds number; see, e.g., \cite{Di-Pietro.Krell:17}, where convergence by compactness is also proved without any small data requirement.

For a given integer $k\ge 0$, the HHO method proposed here hinges on discrete velocity unknowns that are vector-valued polynomials of total degree $\le k$ on mesh elements and faces, and pressure unknowns that are discontinuous polynomials of total degree $\le k$ on the mesh.
Based on these discrete unknowns, three relevant quantities are reconstructed inside each element:
a velocity one degree higher than the discrete unknowns;
a discrete advective derivative;
a discrete divergence whose composition with the local interpolator coincides with the $L^2$-orthogonal projection of the continuous divergence.
The use of the high-order velocity reconstruction allows one to obtain the ${\cal O}(h^{k+1})$-scaling for the viscous term typical of HHO methods.
The use of the discrete advective derivative together with a face-element upwind stabilization in the advective contribution, on the other hand, warrants a robust behaviour in the advection-dominated regime.
In particular, the contribution to the discretization error stemming from the advective term has optimal scaling varying from ${\cal O}(h^{k+1})$ in the diffusion-dominated regime to ${\cal O}(h^{k+\frac12})$ in the advection-dominated regime.
Finally, the discrete divergence operator is designed so as to be surjective in the discrete pressure space, so that an inf-sup condition is verified.

The rest of the paper is organized as described hereafter.
In Section \ref{sec:continuous.problem} we formulate the continuous problem along with the assumptions on the data.
In Section \ref{sec:discrete.setting} we establish the discrete setting (mesh, notation, basic results).
Section \ref{sec:discrete.problem} contains the formulation of the discrete problem preceeded by the required ingredients, the statements of the main results corresponding to Theorems \ref{thm:well-posedness} and \ref{thm:err.est}, and numerical examples.
Section \ref{sec:proofs} contains the proofs of the main results.
Finally, the flux formulation of the discrete problem is discussed in Appendix \ref{sec:flux.formulation}.


\section{Continuous problem}\label{sec:continuous.problem}

Let $\Omega\subset\Real^d$, $d\in\{2,3\}$, denote an open bounded connected polytopal set with Lipschitz boundary $\partial\Omega$, $\vf\in L^2(\Omega)^d$ a volumetric body force, $\visc\in\Real_+^*$ (with $\Real_+^*$ denoting the set of strictly positive real numbers) the dynamic viscosity, $\vel\in{\rm Lip}(\Omega)^d$ a given velocity field such that $\DIV\vel=0$, and $\reac\in\Real_+^*$ a reaction coefficient.
We consider the Oseen problem that consists in seeking the velocity field $\vu:\Omega\to\Real^d$ and the pressure field $p:\Omega\to\Real$ such that
\begin{subequations}\label{eq:strong}
  \begin{alignat}{2}
    \label{eq:strong:momentum}
    -\visc\LAPL\vu + (\vel\SCAL\GRAD)\vu + \reac\vu + \GRAD p &= \vf &\qquad& \text{in $\Omega$,} \\
    \label{eq:strong:mass}
    \DIV\vu &= 0 &\qquad& \text{in $\Omega$,} \\
    \label{eq:strong:bc}
    \vu &= \vec{0}  &\qquad& \text{on $\partial\Omega$,} \\
    \int_\Omega p &= 0.
  \end{alignat}
\end{subequations}
In what follows, the coefficients $\visc$, $\vel$, $\reac$ together with the source term $\vf$ are collectively referred to as the problem data.
\begin{remark}[Reaction coefficients]
  The assumption $\reac>0$ can be relaxed, but we keep it here to simplify some of the arguments in the analysis.
  We are, however, not concerned with the reaction-dominated regime.
  We notice, in passing, that assuming $\reac>0$ brings us closer to the
  unsteady case, where reaction-like terms stem from the discretization of the
  time derivative; see also Remark \ref{rem:improved.pressure.estimate} on this point.
  The case $\reac=0$ is considered in the numerical examples of Section \ref{sec:numerical.examples}.
\end{remark}

Weak formulations for problem \eqref{eq:strong} are classical.
Denote by $H_0^1(\Omega)$ the space of functions that are square-integrable on $\Omega$ along with their first weak derivatives and that vanish on $\partial\Omega$ in the sense of traces, and by $L^2_0(\Omega)$ the space of functions that are square-integrable and have zero mean value on $\Omega$.
For any $X\subset\Omega$, we denote by $({\cdot},{\cdot})_X$ the usual inner product of $L^2(X)$ and by $\|{\cdot}\|_X$ the corresponding norm, and we adopt the convention that the subscript is omitted whenever $X=\Omega$.
The same notations are used for the spaces of vector- and tensor-valued functions $L^2(X)^d$ and $L^2(X)^{d\times d}$, respectively.
Setting $\vU\coloneq H_0^1(\Omega)^d$ and $P\coloneq L^2_0(\Omega)$, a weak formulation for problem \eqref{eq:strong} reads:
Find $(\vu,p)\in \vU\times P$ such that
\begin{equation}\label{eq:weak}
  \begin{alignedat}{2}
    a(\vu,\vv) + b(\vv,p)
    &= (\vf,\vv) &\qquad& \forall\vv\in\vU,
    \\
    -b(\vu,q) &= 0 &\qquad& \forall q\in P,
  \end{alignedat}
\end{equation}
with bilinear forms $a:\vU\times\vU\to\Real$ and $b:\vU\times P\to\Real$ such that
\begin{equation}\label{eq:a.b}
  a(\vu,\vv)\coloneq\visc(\GRAD\vu,\GRAD\vv) + ((\vel\SCAL\GRAD)\vu,\vv) + \reac(\vu,\vv),\quad
  b(\vv,q)\coloneq-(\DIV\vv,q).
\end{equation}


\section{Discrete setting}\label{sec:discrete.setting}

We consider here polygonal or polyhedral meshes corresponding to couples $\Mh\coloneq(\Th,\Fh)$, where $\Th$ is a finite collection of polygonal elements $T$ of maximum diameter equal to $h>0$, while $\Fh$ is a finite collection of hyperplanar faces $F$.
It is assumed henceforth that the mesh $\Mh$ matches the geometrical requirements detailed in \cite[Definition 7.2]{Droniou.Eymard.ea:17}; see also~\cite[Section 2]{Di-Pietro.Tittarelli:17}.
For every mesh element $T\in\Th$, we denote by $\Fh[T]$ the subset of $\Fh$ containing the faces that lie on the boundary $\partial T$ of $T$.
For each face $F\in\Fh[T]$, $\normal_{TF}$ is the (constant) unit normal
vector to $F$ pointing out of $T$, and we denote by $\normal_T$ the piecewise
constant field on $\Fh[T]$ such that $\normal_{T|F}=\normal_{TF}$ for all
$F\in\Fh[T]$.
Boundary faces lying on $\partial\Omega$ and internal faces contained in
$\Omega$ are collected in the sets $\Fhb$ and $\Fhi$, respectively.

Our focus is on the so-called $h$-convergence analysis, so we consider a sequence of refined meshes that is regular in the sense of~\cite[Definition~3]{Di-Pietro.Tittarelli:17}.
The corresponding positive regularity parameter, uniformly bounded away from zero, is denoted by $\varrho$.
The mesh regularity assumption implies, among other, that the diameter $h_T$ of a mesh element $T\in\Th$ is uniformly comparable to the diameter $h_F$ of each face $F\in\Fh[T]$, and that the number of faces in $\Fh[T]$ is bounded uniformly in $h$.

The construction underlying HHO methods hinges on projectors on local polynomial spaces.
Let $X$ denote a mesh element or face.
For a given integer $l\ge 0$, we denote by $\Poly{l}(X)$ the space spanned by the restriction to $X$ of $d$-variate polynomials of total degree $\le l$.
The local $L^2$-orthogonal projector $\lproj[X]{l}:L^2(X)\to\Poly{l}(X)$ is defined as follows:
For all $v\in L^2(X)$, the polynomial $\lproj[X]{l}v\in\Poly{l}(X)$ satisfies
\begin{equation}\label{eq:lproj}
  (\lproj[X]{l}v-v,w)_X=0\qquad\forall w\in\Poly{l}(X).
\end{equation}
The vector version of the $L^2$-projector, denoted by $\vlproj[X]{l}$, is obtained by applying $\lproj[X]{l}$ component-wise.
At the global level, we denote by $\Poly{l}(\Th)$ the space of broken polynomials on $\Th$ whose restriction to every mesh element $T\in\Th$ lies in $\Poly{l}(T)$.
The corresponding global $L^2$-orthogonal projector $\lproj{l}:L^2(\Omega)\to\Poly{l}(\Th)$ is such that, for all $v\in L^2(\Omega)$,
\begin{equation}\label{eq:lprojh}
  (\lproj{l}v)_{|T} \coloneq \lproj[T]{l}v_{|T}.
\end{equation}
Also in this case, the vector version $\vlproj{l}$ is obtained by applying $\lproj{l}$ component-wise.
Broken polynomial spaces are a special instance of the broken Sobolev spaces:
For an integer $m\ge 0$, $$H^m(\Th)\coloneq\left\{v\in L^2(\Omega)\ST v_{|T}\in H^m(T)\quad\forall T\in\Th\right\}.$$
Broken Sobolev spaces will be used to formulate the regularity assumptions on the exact solution required to derive error estimates.

Let now a mesh element $T\in\Th$ be given.
The local elliptic projector $\eproj[T]{l}:H^1(T)\to\Poly{l}(T)$ is defined as follows:
For all $v\in H^1(T)$, the polynomial $\eproj[T]{l}v\in\Poly{l}(T)$ satisfies
\begin{equation}\label{eq:eproj}
  (\GRAD(\eproj[T]{l}v-v),\GRAD w)_T=0\mbox{ for all } w\in\Poly{l}(T)
  \mbox{ and } (\eproj[T]{l}v-v,1)_T=0.
\end{equation}
The vector version $\veproj[T]{l}$ is again obtained by applying $\eproj[T]{l}$ element-wise.
We leave it to the reader to check that both the local $L^2$-orthogonal and elliptic projectors are linear, onto, and idempotent (hence, they map polynomials of degree $\le l$ onto themselves).

To avoid the profileration of generic constants, throughout the rest of the paper the notation $a\lesssim b$ means $a\le Cb$ with real number $C>0$ independent of the meshsize $h$, of the problem data and, for local inequalities on a mesh element or face $X$, also on $X$. The notation $a\simeq b$ means $a\lesssim b\lesssim a$.
When useful, the dependence of the hidden constant is further specified.

On regular mesh sequences, both $\lproj[T]{l}$ and $\eproj[T]{l}$ have optimal approximation properties in $\Poly{l}(T)$, as summarized by the following result (for a proof, see Theorems 1.1, 1.2, and Lemma 3.1 in~\cite{Di-Pietro.Droniou:17*1}):
For $\xi\in\{0,1\}$ and any $s\in\{\xi,\ldots,l+1\}$, it holds for all $T\in\Th$, and all $v\in H^s(T)$, 
\begin{subequations}\label{eq:approx.approx.trace}
  \begin{equation}\label{eq:approx}
    \seminorm[H^m(T)]{v - \pi_T^{\xi,l} v }
    \lesssim h_T^{s-m} 
    \seminorm[H^s(T)]{v}
    \qquad \forall m \in \{0,\ldots,s\},
  \end{equation}
  and, if $s\ge 1$,
  \begin{equation}\label{eq:approx.trace}
    \seminorm[{H^m(\Fh[T])}]{v - \pi_T^{\xi,l} v}
    \lesssim h_T^{s-m-\frac12} 
    \seminorm[H^s(T)]{v}
    \qquad \forall m \in \{0,\ldots,s-1\},
  \end{equation}
  where $H^m(\Fh[T])\coloneq\left\{ v\in L^2(\partial T)\ST v_{|F}\in H^m(F)\quad\forall F\in\Fh[T]\right\}$ is the broken Sobolev space on the boundary of $T$.
\end{subequations}
The hidden constants in \eqref{eq:approx.approx.trace} depend only on $d$, $\varrho$, $\xi$, $l$, and $s$.

To close this section, we note the following local trace and inverse
inequalities (cf.~\cite[Lemmas 1.46 and 1.44]{Di-Pietro.Ern:12}):
For all $T\in\Th$ and all $v\in\Poly{k}(T)$,
\begin{equation}\label{eq:trace.inv}
  \norm[F]{v} \lesssim h_F^{- \frac12} \norm[T]{v}\mbox{ for all $F\in\Fh[T]$ and }
  \norm[T]{\GRAD v}\lesssim h_T^{-1}\norm[T]{v}.
\end{equation}


\section{Discrete problem}\label{sec:discrete.problem}

In this section we introduce the main ingredients of the HHO construction,
formulate the discrete problem, state the main results, and provide some
numerical examples.

\subsection{Discrete unknowns}
Let an integer $k\ge 0$ be fixed.
We define the following space of discrete velocity unknowns on $\Th$, which consist of vector-valued polynomial functions of total degree $\le k$ inside each mesh element and on each mesh face:
$$
\Uh\coloneq\left\{
\uvv\coloneq((\vv[T])_{T\in\Th},(\vv[F])_{F\in\Fh})\ST
\vv[T]\in\Poly{k}(T)^d\mbox{ for all } T\in\Th\text{ and }
\vv[F]\in\Poly{k}(F)^d\mbox{ for all } F\in\Fh
\right\}.
$$
For all $\uvv\in\Uh$, we define the broken polynomial function $\vv[h]\in\Poly{k}(\Th)^d$ obtained by patching element unknowns, i.e.,
$$
\vv[h|T]\coloneq\vv[T]\qquad\forall T\in\Th.
$$
The global interpolator $\Ih:H^1(\Omega)^d\to\Uh$ is such that, for any $\vv\in H^1(\Omega)^d$,
$$
\Ih\vv\coloneq((\vlproj[T]{k}\vv[|T])_{T\in\Th},(\vlproj[F]{k}\vv[|F])_{F\in\Fh}).
$$
For any mesh element $T\in\Th$, we denote by $\UT$ and $\IT$, respectively, the restrictions of $\Uh$ and $\Ih$ to $T$, that is
$$
\UT\coloneq\left\{
\uvv[T]\coloneq(\vv[T],(\vv[F])_{F\in\Fh[T]})\ST
\vv[T]\in\Poly{k}(T)^d\text{ and }
\vv[F]\in\Poly{k}(F)^d\mbox{ for all } F\in\Fh[T]
\right\}
$$
and, for any $\vv\in H^1(T)^d$,
$$
\IT\vv\coloneq(\vlproj[T]{k}\vv,(\vlproj[F]{k}\vv[|F])_{F\in\Fh[T]}).
$$

The HHO scheme is based on the following discrete spaces for the velocity and the pressure which strongly incorporate, respectively, the homogeneous boundary condition on the velocity and the zero-mean value constraint on the pressure:
\begin{equation*}
  \UhD\coloneq\left\{
  \uvv\in\Uh\ST \vv[F]=\vec{0}\quad\forall F\in\Fhb
  \right\},\qquad
  \Ph\coloneq\Poly{k}(\Th)\cap P.
\end{equation*}


\subsection{Viscous bilinear form}

Let a mesh element $T\in\Th$ be fixed.
We define the local velocity reconstruction operator $\rT:\UT\to\Poly{k+1}(T)^d$ such that, for a given $\uvv[T]\in\UT$, $\rT\uvv[T]$ satisfies
\begin{equation}
  \label{eq:rT}
  \begin{alignedat}{2} 
    (\GRAD(\rT\uvv[T]),\GRAD\vw)_T &=
    -(\vv[T],\LAPL\vw)_T + \sum_{F\in\Fh[T]}(\vv[F],\GRAD\vw\normal_{TF})_F
    &\qquad&\forall\vw\in\Poly{k+1}(T)^d,
    \\ 
    \int_T\rT\uvv[T] &= \int_T\vv[T].
  \end{alignedat}
\end{equation}
The above definition can be justified observing that, for all $\vv\in H^1(T)^d$,
\begin{equation}\label{eq:rTIT=veproj}
  \rT\IT\vv=\veproj{k+1}\vv,
\end{equation}
as can be easily verified writing \eqref{eq:rT} with $\uvv[T]$ replaced by
$\IT\vv$ and using the definition \eqref{eq:eproj} of the elliptic projector
with $l=k+1$ in the left-hand side and \eqref{eq:lproj} of the
$L^2$-orthogonal projectors on $T$ and its faces in the right-hand side.

The discrete viscous bilinear form $\mathrm{a}_{\visc,h}:\Uh\times\Uh\to\Real$ is assembled element-wise as follows:
\begin{subequations}\label{eq:ah.visc}
\begin{equation}\label{eq:ah.visc.sumT}
  \mathrm{a}_{\visc,h}(\uvw,\uvv)=\sum_{T\in\Th}\mathrm{a}_{\visc,T}(\uvw[T],\uvv[T])
\end{equation}
where, for any $T\in\Th$, the local bilinear form $\mathrm{a}_{\visc,T}:\UT\times\UT\to\Real$ is such that
\begin{equation}\label{eq:aT.visc}
  \mathrm{a}_{\visc,T}(\uvw[T],\uvv[T])
  \coloneq\visc(\GRAD(\rT\uvw[T]),\GRAD(\rT\uvv[T]))_T + \mathrm{s}_{\visc,T}(\uvw[T],\uvv[T]).
\end{equation}
The first contribution in the right-hand side is responsible for consistency, whereas the second is a stabilization bilinear form which we can take such that
\begin{equation}\label{eq:sT.visc}
  \mathrm{s}_{\visc,T}(\uvw[T],\uvv[T])\coloneq
  \sum_{F\in\Fh[T]}\frac{\visc}{h_F}((\dTF-\dT)\uvw[T],(\dTF-\dT)\uvv[T])_F
\end{equation}
\end{subequations}
where, for any $\uvv[T]\in\UT$, we have introduced the difference operators
such that for any $\uvv[T]\in\UT$,
\begin{equation}\label{eq:dT.dTF}
  (\dT\uvv[T],(\dTF\uvv[T])_{F\in\Fh[T]})\coloneq\IT(\rT\uvv[T])-\uvv[T]\in\UT.
\end{equation}
Using \eqref{eq:rTIT=veproj} together with the linearity and idempotency of the $L^2$-orthogonal projectors on mesh elements and faces, it can be proved that the following polynomial consistency property holds (see, e.g., \cite[Section 3.1.4]{Di-Pietro.Tittarelli:17} for the details):
For all $\vw\in\Poly{k+1}(T)^d$,
\begin{equation}\label{eq:dT.dTF.consistency}
  (\dT\IT\vw,(\dTF\IT\vw)_{F\in\Fh[T]})=\underline{\vec{0}}\in\UT.
\end{equation}

\begin{remark}[Viscous stabilization bilinear form]\label{rem:sT.visc}
  More general viscous stabilization bilinear forms can be considered.
  Following \cite[Section 5.3]{Boffi.Di-Pietro:17}, the following set of sufficient design conditions on $\mathrm{s}_{\visc,T}$ ensure that the required stability and consistency properties for $\mathrm{a}_{\visc,h}$ hold:
  \begin{compactenum}[\rm (S1)]
  \item \emph{Symmetry and positivity.} $\mathrm{s}_{\visc,T}$ is symmetric and positive semidefinite.
  \item \emph{Stability and boundedness.} It holds for all $\uvv[T]\in\UT$,
    \begin{equation*} 
      \visc\norm[1,T]{\uvv[T]}^2\simeq\mathrm{a}_{\visc,T}(\uvv[T],\uvv[T])\mbox{ where }
      \norm[1,T]{\uvv[T]}^2\coloneq
      \norm[T]{\GRAD\vv[T]}^2+\sum_{F\in\Fh[T]}h_F^{-1}\norm[F]{\vv[F]-\vv[T]}^2.
    \end{equation*}
    Summing over $T\in\Th$, this implies in particular that it holds, for all $\uvv\in\Uh$,
    \begin{equation}\label{eq:ah.visc:stability}
      \visc\norm[1,h]{\uvv}^2\simeq\mathrm{a}_{\visc,h}(\uvv,\uvv)\mbox{ where }
      \norm[1,h]{\uvv}^2\coloneq\sum_{T\in\Th}\norm[1,T]{\uvv[T]}^2.
    \end{equation}
  \item \emph{Polynomial consistency.} For all $\vw\in\Poly{k+1}(T)^d$ and all $\uvv[T]\in\UT$, it holds that
    $$
    \mathrm{s}_{\visc,T}(\IT\vw,\uvv[T])=0.
    $$
  \end{compactenum}
  The stabilization bilinear form \eqref{eq:sT.visc} is clearly symmetric and positive semidefinite, and thus it satisfies (S1).
  A proof of (S2) can be found in \cite[Lemma 4]{Di-Pietro.Ern.ea:14}, where the scalar case is considered.
  Finally, (S3) is an immediate consequence of \eqref{eq:dT.dTF.consistency}.
  To close this remark, we note the following important consequence of (S1)--(S3):
  For any $\vw\in H^{k+2}(T)^d$,
  \begin{equation}\label{eq:sT.visc:consistency}
    \mathrm{s}_{\visc,T}(\IT\vw,\IT\vw)^{\frac12}\lesssim \visc h_T^{k+1}\seminorm[H^{k+2}(T)^d]{\vw}.
  \end{equation}
\end{remark}

\begin{remark}[Comparison with the LDG-H stabilization]\label{rem:HDG.stab}
  The extension to the vector case of the LDG-H stabilization originally introduced in \cite{Castillo.Cockburn.ea:00,Cockburn.Gopalakrishnan.ea:10} for scalar diffusion problems reads
  $$ 
    \mathrm{s}_{\visc,T}^{\rm ldg}(\uvw[T],\uvv[T])
    =\sum_{F\in\Fh[T]}\eta\visc(\vw[F]-\vw[T],\vv[F]-\vv[T])_F,
  $$ 
  where $\eta>0$ is a user-defined stabilization parameter.
  The main difference with respect to the HHO stabilization defined by \eqref{eq:sT.visc} is that $\mathrm{s}_{\visc,T}^{\rm ldg}$ does not satisfy property (S3).
  In particular, when using $\eta=h_F^{-1}$, the consistency estimate \eqref{eq:sT.visc:consistency} modifies to
  $$
  \mathrm{s}_{\visc,T}^{\rm ldg}(\IT\vw,\IT\vw)^{\frac12}\lesssim \visc h_T^k\seminorm[H^{k+1}(T)^d]{\vw}.
  $$
  As a result, up to one order of convergence is lost in the error estimate.
  We refer to \cite{Cockburn.Di-Pietro.ea:16} for further details including a discussion on possible fixes.
\end{remark}


\subsection{Advection-reaction bilinear form}

Let a mesh element $T\in\Th$ be fixed and set, for the sake of brevity,
$$
\svel\coloneq\vel_{|F}\SCAL\normal_{TF}\mbox{ for all }F\in\Fh[T].
$$
We define the local advective derivative reconstruction $\Gb:\UT\to\Poly{k}(T)^d$ such that, for all $\uvv[T]\in\UT$ and all $\vw\in\Poly{k}(T)^d$,
\begin{equation}\label{eq:Gb}
  (\Gb\uvv[T],\vw)_T
  =((\vel\SCAL\GRAD)\vv[T],\vw)_T + \sum_{F\in\Fh[T]}(\svel(\vv[F]-\vv[T]),\vw)_F.
\end{equation}
The global advection-reaction bilinear form $a_{\vel,\reac,h}:\Uh\times\Uh\to\Real$ is assembled element-wise as follows:
\begin{subequations}\label{eq:ah.vel.reac}
  \begin{equation}\label{eq:ah.vel.reac:sumT}
    \mathrm{a}_{\vel,\reac,h}(\uvw,\uvv)
    \coloneq\sum_{T\in\Th}\mathrm{a}_{\vel,\reac,T}(\uvw[T],\uvv[T])
  \end{equation}
  where, for all $T\in\Th$, the local bilinear form $\mathrm{a}_{\vel,\reac,T}:\UT\times\UT\to\Real$ is such that
\begin{equation}\label{eq:aT.vel.reac}
  \mathrm{a}_{\vel,\reac,T}(\uvw[T],\uvv[T])  
  \coloneq
  -(\vw[T],\Gb\uvv[T])_T + \reac(\vw[T],\vv[T])_T + \mathrm{s}^-_{\vel,T}(\uvw[T],\uvv[T]).
\end{equation}
Here, letting $\xi^\pm\coloneq\frac{|\xi|\pm\xi}{2}$ for any $\xi\in\Real$, we have set
\begin{equation}\label{eq:sT.vel}
  \mathrm{s}^\pm_{\vel,T}(\uvw[T],\uvv[T])\coloneq
  \sum_{F\in\Fh[T]}(\svel^\pm(\vw[F]-\vw[T]),\vv[F]-\vv[T])_F.
\end{equation}
\end{subequations}
\begin{remark}[Reformulation of the advective-reactive bilinear form]
  It can be checked using the definition \eqref{eq:Gb} of $\Gb$, the regularity of $\vel$, the single-valuedness of interface unknowns, and the strongly enforced boundary condition that it holds, for all $\uvw,\uvv\in\UhD$,
  \begin{equation}\label{eq:ah.vel.reac:ter}
    \mathrm{a}_{\vel,\reac,h}(\uvw,\uvv)
    = \sum_{T\in\Th}\left(
    (\Gb\uvw[T],\vv[T])_T + \reac(\vw[T],\vv[T])_T + \mathrm{s}_{\vel,T}^+(\uvw[T],\uvv[T])
    \right).
  \end{equation}
  Summing \eqref{eq:ah.vel.reac:sumT} and \eqref{eq:ah.vel.reac:ter} and dividing by two, we arrive at the following equivalent expression:
  \begin{equation}\label{eq:ah.vel.reac:bis}
    \begin{aligned}
      \mathrm{a}_{\vel,\reac,h}(\uvw,\uvv)
      =& \sum_{T\in\Th}\left(
      \frac12(\Gb\uvw[T],\vv[T])_T
      - \frac12(\vw[T],\Gb\uvv[T])_T
      \right)
      +\sum_{T\in\Th}\reac(\vw[T],\vv[T])_T
      \\
      &+ \sum_{T\in\Th}\sum_{F\in\Fh[T]}(\frac{|\svel|}{2}(\vw[F]-\vw[T]),\vv[F]-\vv[T])_F.
    \end{aligned}
  \end{equation}
  This reformulation of $\mathrm{a}_{\vel,\reac,h}$ shows that (i) the consistent contribution in the advective term, corresponding to the first addend in the right-hand side of \eqref{eq:ah.vel.reac:bis}, is skew-symmetric. As a result, it does not contribute to the global kinetic energy balance obtained by setting $\uvv=\uvu$ in \eqref{eq:discrete:momentum} below; (ii) the upwind stabilization can in fact be interpreted as a least-square penalization of face-element differences. A similar interpretation in the context of DG methods was discussed in \cite{Brezzi.Marini.ea:04}.
\end{remark}
\begin{remark}[Advective stabilization bilinear form]

  \newcommand{\absA}{\widehat{A}}
  
  Following \cite[Section 4.2]{Di-Pietro.Droniou.ea:15}, in \eqref{eq:aT.vel.reac}
  we can consider the following more general stabilization bilinear form:
  $$
  \mathrm{s}^-_{\vel,T}(\uvw[T],\uvv[T])=
  \sum_{F\in\Fh[T]}(\frac{\visc}{h_F}A^-(\Pe_{TF})(\vw[F]-\vw[T]),\vv[F]-\vv[T])_F,
  $$
  where, for all $T\in\Th$ and all $F\in\Fh[T]$, the local face P\'eclet number is such that
  \begin{equation}\label{eq:PeTF}
    \Pe_{TF}\coloneq\frac{\svel h_T}{\visc},
  \end{equation}
  while the function $A^-:\Real\to\Real$ is such that $A^-(s)=\frac12(\absA(s)-s)$ with $\absA:\Real\to\Real$ matching the following sufficient design conditions:
  \begin{compactenum}[\rm ({A}1)]
  \item \emph{Lipschitz continuity, positivity, and symmetry.} $\absA$ is a Lipschitz-continuous function such that $\absA(0)=0$ and, for all $s\in\Real$, $\absA(s)\ge 0$ and $\absA(-s)=\absA(s)$.
  \item \emph{Growth.} There exists $C_A\ge 0$ such that $\absA(s)\ge C_A |s|$ for all $|s|\ge 1$.
  \end{compactenum}
  Besides the upwind stabilization \eqref{eq:sT.vel}, notable examples of stabilizations that match the above design conditions include the locally upwinded $\theta$-scheme and the Scharfetter--Gummel scheme.
\end{remark}

%

\subsection{Velocity-pressure coupling}

Let a mesh element $T\in\Th$ be fixed, and define the discrete divergence operator $\DT:\UT\to\Poly{k}(T)$ such that, for all $\uvv[T]\in\UT$ and all $q\in\Poly{k}(T)$,
\begin{equation}\label{eq:DT}
  (\DT\uvv[T],q)_T\coloneq-(\vv[T],\GRAD q)_T + \sum_{F\in\Fh[T]}(\vv[F]\SCAL\normal_{TF},q)_F.
\end{equation}
For any $T\in\Th$ and any $\vv\in H^1(T)^d$, writing \eqref{eq:DT} for $\uvv[T]=\IT\vv$ and using the definitions \eqref{eq:lproj} of the $L^2$-orthogonal projectors on $T$ and its faces, we infer that
\begin{equation}\label{eq:DT:commuting}
  \DT\IT\vv = \lproj[T]{k}(\DIV\vv).
\end{equation}
We define the velocity-pressure coupling bilinear form $\mathrm{b}_h:\Uh\times\Ph\to\Real$ such that
\begin{equation}\label{eq:bh}
  \mathrm{b}_h(\uvv,q_h)\coloneq -\sum_{T\in\Th}(\DT\uvv[T],q_h)_T.
\end{equation}


\subsection{Reference quantities, P\'eclet numbers, and discrete norms}

For any mesh element $T\in\Th$, we define the following local reference velocity and time:
\begin{equation}\label{eq:vref.tref}
    \vref\coloneq\norm[L^\infty(T)^d]{\vel},\qquad
    \tref\coloneq\frac{1}{\max(\reac,L_{\vel,T})}\mbox{ with }
    L_{\vel,T}\coloneq\max_{1\le i\le d}\norm[L^\infty(T)^d]{\GRAD\svel[i]},
\end{equation}
as well as the following local P\'eclet number (see \eqref{eq:PeTF} for the definition of $\Pe_{TF}$):
\begin{equation}\label{eq:PeT}
  \Pe_T\coloneq\max_{F\in\Fh[T]}\norm[L^\infty(F)]{\Pe_{TF}}.
\end{equation}
In the discussion, we will also need the following global reference time $\tref[h]$ and P\'eclet numbers $\Pe_h$ and $\Pe_\Omega$:
\begin{equation}\label{eq:Peh.trefh}
  \tref[h]\coloneq\min_{T\in\Th}\tref,\qquad
  \Pe_h\coloneq\max_{T\in\Th}\Pe_T,\qquad
  \Pe_\Omega\coloneq\frac{\norm[L^\infty(\Omega)^d]{\vel} d_\Omega}{\visc},
\end{equation}
where we have denoted by $d_\Omega$ the diameter of $\Omega$.
We equip the discrete pressure space $\Ph$ with the $L^2$-norm and the
discrete velocity space $\UhD$ with the following energy norm:
\begin{equation}\label{eq:normUh}
  \norm[\vU,h]{\uvv}
  \coloneq\left(
  \norm[\visc,h]{\uvv}^2 + \norm[\vel,\reac,h]{\uvv}^2
  \right)^{\frac12},
\end{equation}
where we have set
\begin{equation}\label{eq:norm.visc_vel.reac}
  \norm[\visc,h]{\uvv}^2\coloneq\mathrm{a}_{\visc,h}(\uvv,\uvv)\mbox{ and }
  \norm[\vel,\reac,h]{\uvv}^2
  \coloneq\sum_{T\in\Th}\left(
  \frac12\sum_{F\in\Fh[T]}\norm[F]{|\svel|^{\frac12}(\vv[F]-\vv[T])}^2+\tref^{-1}\norm[T]{\vv[T]}^2
  \right).
\end{equation}
Given a linear functional $\mathfrak{f}$ on $\UhD$, its dual norm is given by
\begin{equation}\label{eq:normUh.dual}
  \norm[\vU^*,h]{\mathfrak{f}}
  \coloneq\sup_{\uvv\in\UhD\setminus\{\underline{\vec{0}}\}}\frac{|\langle\mathfrak{f},\uvv[h]\rangle|}{\norm[\vU,h]{\uvv}}.
\end{equation}


\subsection{Discrete problem and main results}\label{sec:discrete:main.results}

Introducing the diffusion-advection-reaction bilinear form
\begin{equation}\label{eq:ah}
  \mathrm{a}_h\coloneq\mathrm{a}_{\visc,h}+\mathrm{a}_{\vel,\reac,h},
\end{equation}
with $\mathrm{a}_{\visc,h}$ defined by \eqref{eq:ah.visc} and
$\mathrm{a}_{\vel,\reac,h}$ by \eqref{eq:ah.vel.reac},
the HHO scheme for problem~\eqref{eq:strong} reads:
Find $(\uvu,p_h)\in\UhD\times\Ph$ such that
\begin{subequations}\label{eq:discrete}
  \begin{alignat}{2}
    \label{eq:discrete:momentum}
    \mathrm{a}_h(\uvu,\uvv) + \mathrm{b}_h(\uvv,p_h) &= (\vf,\vv[h]) &\qquad& \forall\uvv\in\UhD,
    \\
    \label{eq:discrete:mass}
    -\mathrm{b}_h(\uvu,q_h) &= 0 &\qquad& \forall q_h\in\Ph.
  \end{alignat}
\end{subequations}
In the rest of this section, we state and comment the main results of the analysis, whose proofs are postponed to Section \ref{sec:proofs}.
The well-posedness of problem \eqref{eq:discrete} is studied in the following theorem.
\begin{theorem}[Well-posedness]\label{thm:well-posedness}
  The following holds:
  \begin{compactenum}[(i)]
  \item \emph{Coercivity of $\mathrm{a}_h$.}
    It holds for all $\uvv\in\UhD$,
    \begin{equation}\label{eq:ah:stability}
      \Ca\norm[\vU,h]{\uvv}^2
      \lesssim\mathrm{a}_h(\uvv,\uvv)\mbox{ with }
      \Ca\coloneq\min_{T\in\Th}\left(1,\tref\reac\right).
    \end{equation}
  \item \emph{Inf-sup condition on $\mathrm{b}_h$.} For all $q_h\in\Ph$, it holds that
    \begin{equation}\label{eq:bh:stability}
      \Cb\norm{q_h}\lesssim\sup_{\uvv[h]\in\UhD\setminus\{\underline{\vec{0}}\}}
      \frac{\mathrm{b}_h(\uvv[h],q_h)}{\norm[\vU,h]{\uvv}}\mbox{ with }
      \Cb\coloneq\left[\visc(1+\Pe_h) + \tref[h]^{-1}\right]^{-\frac12}.
    \end{equation}
  \item \emph{Continuity of $\mathrm{a}_h$.} It holds for all $\uvw,\uvv\in\UhD$,
    \begin{equation}\label{eq:ah:continuity}
      |\mathrm{a}_h(\uvw,\uvv)|\lesssim (1+\Pe_\Omega)\norm[\vU,h]{\uvw}\norm[\vU,h]{\uvv}.
    \end{equation}
  \end{compactenum}
  As a consequence, problem \eqref{eq:discrete} is well-posed and the following a priori bounds hold:
  \begin{equation}\label{eq:a-priori bounds}
    \norm[\vU,h]{\uvu}\lesssim\frac1\Ca \visc^{-\frac12}\norm{\vf},\qquad
    \norm{p_h}\lesssim\frac1\Cb\left(1+\frac{1+\Pe_\Omega}{\Ca}\right)
    \visc^{-\frac12}\norm{\vf}.
  \end{equation}
\end{theorem}
\begin{proof}
  See Section \ref{sec:proofs:well-posedness}.
\end{proof}
We next investigate the convergence of the method.
We measure the error as the difference between the discrete solution and the interpolant of the exact solution defined as
\begin{equation*} 
  (\uvhu,\hp)\coloneq(\Ih\vu,\lproj{k}p)\in\UhD\times\Ph.
\end{equation*}
Upon observing that, as a consequence of \eqref{eq:DT:commuting}, $\mathrm{b}_h(\uvhu,q_h)=-(\lproj{k}(\DIV\vu),q_h)=-(\DIV\vu,q_h)=0$, straightforward manipulations show that the discretization error
$$
(\uve,\epsilon_h)\coloneq (\uvu,p_h) - (\uvhu,\hp)
$$
solves the following problem:
\begin{subequations}\label{eq:discrete.error}
  \begin{alignat}{2}
    \label{eq:discrete.error:momentum}
    \mathrm{a}_h(\uve,\uvv) + \mathrm{b}_h(\uvv,\epsilon_h) &= \langle\mathfrak{R}(\vu,p),\uvv\rangle
    &\qquad&\forall\uvv\in\UhD,
    \\ \label{eq:discrete.error:mass}
    -\mathrm{b}_h(\uve,q_h) &= 0
    &\qquad&\forall q_h\in\Ph,    
  \end{alignat}
\end{subequations}
where $\mathfrak{R}(\vu,p)$ is the residual linear functional on $\UhD$ such that, for all $\uvv\in\UhD$,
\begin{equation}\label{eq:frakR}
  \langle\mathfrak{R}(\vu,p),\uvv\rangle
  \coloneq (\vf,\vv[h]) - \mathrm{a}_h(\uvhu,\uvv) - \mathrm{b}_h(\uvv,\hp).
\end{equation}
\begin{theorem}[Error estimates and convergence]\label{thm:err.est}
  Denote by $(\vu,p)\in\vU\times P$ and by $(\uvu,p_h)\in\UhD\times\Ph$ the unique solutions of the weak \eqref{eq:weak} and discrete \eqref{eq:discrete} problems, respectively.
  Then, recalling the notation of Theorem \ref{thm:well-posedness}, the following abstract error estimates hold:
  \begin{equation}\label{eq:err.est}
    \norm[\vU,h]{\uve}\lesssim\frac1\Ca\norm[\vU^*,h]{\mathfrak{R}(\vu,p)},\qquad
    \norm{\epsilon_h}\lesssim\frac1\Cb\left(1+\frac{1+\Pe_\Omega}{\Ca}\right)\norm[\vU^*,h]{\mathfrak{R}(\vu,p)}.
  \end{equation}
  Moreover, assuming the additional regularity $\vu\in H^{k+2}(\Th)^d$ and $p\in H^1(\Omega)\cap H^{k+1}(\Th)$, it holds that
  \begin{equation}\label{eq:conv.rate}
    \norm[\vU^*,h]{\mathfrak{R}(\vu,p)}
    \lesssim
    \left[\sum_{T\in\Th}\left( h_T^{2(k+1)}\cN[1,T] + \min(1,\Pe_T)h_T^{2k+1}\cN[2,T] \right) \right]^{\frac12},
  \end{equation}
  where, for the sake of brevity, we have defined for all $T\in\Th$ the following bounded quantities:
  $$
  \cN[1,T]\coloneq \visc\seminorm[H^{k+2}(T)^d]{\vu}^2 + \visc^{-1}\seminorm[H^{k+1}(T)]{p}^2,
  \quad
  \cN[2,T]\coloneq\vref\seminorm[H^{k+1}(T)^d]{\vu}^2.
  $$
\end{theorem}
\begin{proof}
  See Section \ref{sec:proofs:convergence}.
\end{proof}
\begin{remark}[Robustness of the a priori estimates \eqref{eq:err.est}]\label{eq:bound.robustness}
  The constant $\Ca$ in the a priori estimate for the velocity is independent of the P\'eclet number, hence the resulting error estimate is robust with respect to the latter.
    Plugging the definition \eqref{eq:vref.tref} of the local reference time into that of $\Ca$ (see \eqref{eq:ah:stability}), one can additionally see that the factor $\frac{1}{\Ca}$ depends adversely on the local Lipschitz module of the velocity, meaning that the accuracy of the results is expected to be lower when the velocity field has abrupt spatial variations.
    
  The multiplicative constant in the a priori estimate for the pressure, on the other hand, depends on the global P\'eclet numbers defined in \eqref{eq:Peh.trefh}, showing that the control of the error on the pressure is not robust for dominant advection (formally corresponding to $\Pe_h\to+\infty$).
\end{remark}
\begin{remark}[Convergence rate]
  Using the local P\'eclet number in \eqref{eq:conv.rate} allows us to establish an estimate on $\norm[\vU^*,h]{\mathfrak{R}(\vu,p)}$  which locally adjusts to the various regimes of \eqref{eq:strong}.
  In mesh elements where diffusion dominates so that $\Pe_T\le h_T$, the contribution to the right-hand side of \eqref{eq:strong} is $\mathcal O(h_T^{2(k+1)})$.
  In mesh elements where advection dominates so that $\Pe_T\ge 1$, on the other hand, the contribution is $\mathcal O(h_T^{2k+1})$.
  The transition region, where $\Pe_T$ is between $h_T$ and $1$, corresponds to intermediate orders of convergence. Notice also that the viscous contribution exhibits the superconvergent behavior $\mathcal O(h_T^{2(k+1)})$ typical of HHO methods, see \cite{Di-Pietro.Ern.ea:14}.
  As a result, the balancing with the advective contribution is slightly different with respect to, e.g., the DG method of \cite{Di-Pietro.Ern.ea:08}, where the viscous contribution scales as $\mathcal O(h_T^{2k})$.
\end{remark}
\begin{remark}[Static condensation]\label{rem:static.cond}
  The size of the linear system corresponding to the discrete problem~\eqref{eq:discrete} can be significantly reduced by resorting to static condensation.
  Following the procedure hinted to in~\cite{Aghili.Boyaval.ea:15} and detailed in~\cite[Section~6.2]{Di-Pietro.Ern.ea:16*1}, it can be shown that the only globally coupled variables are the face unknowns for the velocity and the mean value of the pressure in each mesh element.
  As a result, after statically condensing the other discrete unknowns, the size of the matrix in the left-hand side of the linear system is, denoting by $N_{\cal F}^{\rm i}$ the number of internal faces and by $N_{\cal T}$ the number of mesh elements,
  $$
{k+d-1\choose k} N_{\cal F}^{\rm i} + N_{\cal T}.
  $$
\end{remark}

\begin{remark}[An improved pressure estimate]\label{rem:improved.pressure.estimate}
  The dependence on the global P\'eclet number $\Pe_\Omega$ in the pressure error estimate \eqref{eq:err.est} can be removed by working with two velocity norms, as briefly described hereafter.
    We define on $\UhD$ the augmented norm such that, for all $\uvv\in\Uh$,
    $$
    \tnorm[\vU,h]{\uvv}\coloneq\left(
    \norm[\vU,h]{\uvv}^2 + \sum_{T\in\Th}h_T\vref^{-1}\norm[T]{\Gb\uvv[T]}^2
    \right)^{\frac12},
    $$
    where the last summand is taken only if $\vref\neq 0$.
    We additionally assume that
    \begin{equation}\label{eq:ah.continuity.improved}
      \frac{\vref}{h_T\mu}\le 1\qquad\forall T\in\Th.
    \end{equation}
    This assumption has a straightforward interpretation when considering unsteady problems for which the reaction term stems from the finite difference discretization of a time derivative, and $\mu$ is therefore proportional to the inverse of the discrete time step $\delta t$.
    As a matter of fact, in this case the condition \eqref{eq:ah.continuity.improved} stipulates that, when $\vref\neq 0$, $\delta t\le h_T/\vref$, that is to say, the time step is less than the caracteristic time required to cross the mesh element $T$.
    In the above framework, one can show that
    \begin{inparaenum}[(i)]
    \item the inf-sup condition \eqref{eq:bh:stability} holds with $\norm[\vU,h]{{\cdot}}$ replaced by $\tnorm[\vU,h]{{\cdot}}$, and 
    \item the following improved boundedness can be proved for $\mathrm{a}_h$:
    For all $\uvw,\uvv\in\UhD$, $\mathrm{a}_h(\uvw,\uvv)\lesssim\norm[\vU,h]{\uvw}\tnorm[\vU,h]{\uvv}$. 
    \end{inparaenum}
    As a result, we have the following estimate for the error on the pressure:
    $$
    \begin{aligned}
      \Cb\norm{\epsilon_h}
      &\lesssim\sup_{\uvv\in\UhD\setminus\{\underline{\vec{0}}\}}\frac{\mathrm{b}_h(\uvv,\epsilon_h)}{\tnorm[\vU,h]{\uvv}}
      \\
      &=\sup_{\uvv\in\UhD\setminus\{\underline{\vec{0}}\}}\frac{\langle\mathfrak{R}(\vu,p),\uvv\rangle - \mathrm{a}_h(\uve,\uvv)}{\tnorm[\vU,h]{\uvv}}
      \\
      &\lesssim\tnorm[\vU^*,h]{\mathfrak{R}(\vu,p)} + \norm[\vU,h]{\uve},
    \end{aligned}
    $$
    where we have used the inf-sup condition on $\mathrm{b}_h$ with respect to the augmented norm in the first line, the error equation \eqref{eq:discrete.error:momentum} to pass to the second, and $\tnorm[\vU^*,h]{{\cdot}}$ is the dual norm of $\UhD$ defined as in \eqref{eq:normUh.dual} with $\norm[\vU,h]{{\cdot}}$ replaced by $\tnorm[\vU,h]{{\cdot}}$.
    Under the regularity assumptions of Theorem \ref{thm:err.est}, an inspection of \eqref{eq:err.est:conv.rate:2} reveals that a bound analogous to \eqref{eq:conv.rate} holds for $\tnorm[\vU^*,h]{\mathfrak{R}(\vu,p)}$, so that the error estimate on the pressure becomes
    $$
    \norm{\epsilon_h}\lesssim\frac{1}{\Cb}\left(1+\frac{1}{\Ca}\right)
    \left[\sum_{T\in\Th}\left( h_T^{2(k+1)}\cN[1,T] + \min(1,\Pe_T)h_T^{2k+1}\cN[2,T] \right) \right]^{\frac12},
    $$
    and the global P\'eclet number $\Pe_\Omega$ no longer appears in the multiplicative constant in the right-hand side.
\end{remark}



\subsection{Numerical examples}\label{sec:numerical.examples}
\begin{figure}
  \centering
  \includegraphics[height=3.5cm]{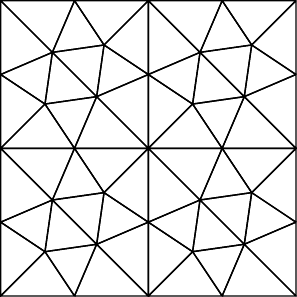}
  \hspace{0.5cm}
  \includegraphics[height=3.5cm]{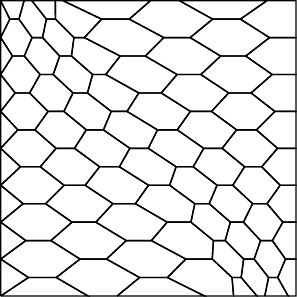}
  \caption{Triangular and hexagonal meshes.}
  \label{fig:meshes}
\end{figure}
In order to confirm the error estimates of Theorem \ref{thm:err.est}, we use the well-known exact solution due to Kovasznay \cite{Kovasznay:48}, which we adapt here to the Oseen setting using the analytical expression of the velocity for the advection field $\vel$.
For a given value $\Pe\in \Real^*_+$ of the P\'eclet number, setting $\lambda \coloneq \Pe - \sqrt{\Pe^2 + 4\pi^2}$, we take
$$
{\vu}_{\lambda}(\vx) \coloneq \big( 1 - \exp( \lambda x ) \cos( 2 \pi y ), \tfrac{\lambda}{2\pi} \exp( \lambda x ) \sin(2\pi y) \big),\qquad
p_{\lambda}(\vx) \coloneq \overline{p}-\tfrac{1}{2}\exp(2\lambda x),
$$
with $\nu \coloneq (2\Pe)^{-1}$, $\vel\coloneq {\vu}_{\lambda}$,
$\reac \coloneq 0$, and $\overline{p}$ chosen such that $(p_{\lambda},1) = 0$.
The computational domain is the square \mbox{$\Omega=(-0.5,1.5)\times(0,2)$}, approximated with refined families of triangular and (predominantly) hexagonal meshes; see Figure~\ref{fig:meshes}.
The former correspond to the mesh family \texttt{mesh1} of the FVCA5 benchmark \cite{Herbin.Hubert:08}, whereas the latter is taken from \cite{Di-Pietro.Lemaire:15}.
Dense and sparse linear algebra are based on the C++ library \texttt{Eigen} \cite{Guennebaud.Jacob.ea:10}.


\begin{figure}
  \centering
  \foreach \Reynolds in {0.01,1,10000} {
    \def\plotordera{1}
      \def\plotorderb{2}
      \def\plotorderc{3}
      \def\plotorderd{4}
      
      \ifnum100<\Reynolds
      \def\plotordera{0.5}
      \def\plotorderb{1.5}
      \def\plotorderc{2.5}
      \def\plotorderd{3.5}
      \fi
    \begin{minipage}{0.45\textwidth}
      \begin{tikzpicture}[scale=0.75]
        \begin{loglogaxis}[
            legend pos=outer north east,
            grid=major]
          \foreach \k in {0,1,2,3}{ \addplot table[x=msh,y=err_Gh] {mesh1_err_Re\Reynolds_k\k.dat}; }
          \logLogSlopeTriangle{0.85}{0.31}{0.1}{\plotordera}{black};
          \logLogSlopeTriangle{0.85}{0.31}{0.1}{\plotorderb}{black};
          \logLogSlopeTriangle{0.85}{0.31}{0.1}{\plotorderc}{black};
          \logLogSlopeTriangle{0.85}{0.31}{0.1}{\plotorderd}{black};
          \legend{$k=0$,$k=1$,$k=2$,$k=3$}
        \end{loglogaxis}
      \end{tikzpicture}
      \subcaption{$\norm[\vU,h]{\uve}$ v. $h$ ($\Pe = \Reynolds$)}
    \end{minipage}
    \hspace{0.25cm}
    \begin{minipage}{0.45\textwidth}
      \begin{tikzpicture}[scale=0.75]
        \begin{loglogaxis}[
            legend pos=outer north east,
            grid=major]
          \foreach \k in {0,1,2,3}{ \addplot table[x=msh,y=err_Ph] {mesh1_err_Re\Reynolds_k\k.dat}; }
          \logLogSlopeTriangle{0.85}{0.31}{0.1}{\plotordera}{black};
          \logLogSlopeTriangle{0.85}{0.31}{0.1}{\plotorderb}{black};
          \logLogSlopeTriangle{0.85}{0.31}{0.1}{\plotorderc}{black};
          \logLogSlopeTriangle{0.85}{0.31}{0.1}{\plotorderd}{black};
          \legend{$k=0$,$k=1$,$k=2$,$k=3$}
        \end{loglogaxis}      
      \end{tikzpicture}
      \subcaption{$\norm{\epsilon_h}$ v. $h$ ($\Pe = \Reynolds$)}
    \end{minipage}
    \bigskip\\
  } 
  \caption{Velocity error (left) and pressure error (right) versus meshsize $h$ on the triangular mesh sequence for $\Pe\in\{10^{-2},1,10^4\}$.\label{fig:err.tri}}
  \captionof{table}{Estimated asymptotic orders of convergence of the relative errors on the triangular mesh sequence.\label{tab:tri}}  
  \begin{tabular}{c|ccc|ccc|ccc}
\toprule
 & \multicolumn{3}{c|}{$\Pe = 0.01$} & \multicolumn{3}{c|}{$\Pe = 1$} & \multicolumn{3}{c}{$\Pe = 10000$} \\ 
\midrule 
& $\norm[\vU,h]{\uve}$ & $\norm{\ve[h]}$ & $\norm{\epsilon_h}$& $\norm[\vU,h]{\uve}$ & $\norm{\ve[h]}$ & $\norm{\epsilon_h}$& $\norm[\vU,h]{\uve}$ & $\norm{\ve[h]}$ & $\norm{\epsilon_h}$\\ \midrule 
$k=0$ & \pgfmathprintnumber[fixed,precision=2,fixed zerofill=true]{0.961118}  & \pgfmathprintnumber[fixed,precision=2,fixed zerofill=true]{1.86289}  & \pgfmathprintnumber[fixed,precision=2,fixed zerofill=true]{1.0682}  & \pgfmathprintnumber[fixed,precision=2,fixed zerofill=true]{0.823513}  & \pgfmathprintnumber[fixed,precision=2,fixed zerofill=true]{1.64695}  & \pgfmathprintnumber[fixed,precision=2,fixed zerofill=true]{1.11301}  & \pgfmathprintnumber[fixed,precision=2,fixed zerofill=true]{0.480068}  & \pgfmathprintnumber[fixed,precision=2,fixed zerofill=true]{0.773591}  & \pgfmathprintnumber[fixed,precision=2,fixed zerofill=true]{1.76023} \\ 
$k=1$ & \pgfmathprintnumber[fixed,precision=2,fixed zerofill=true]{1.91163}  & \pgfmathprintnumber[fixed,precision=2,fixed zerofill=true]{3.02414}  & \pgfmathprintnumber[fixed,precision=2,fixed zerofill=true]{1.94395}  & \pgfmathprintnumber[fixed,precision=2,fixed zerofill=true]{1.82833}  & \pgfmathprintnumber[fixed,precision=2,fixed zerofill=true]{2.70707}  & \pgfmathprintnumber[fixed,precision=2,fixed zerofill=true]{1.95834}  & \pgfmathprintnumber[fixed,precision=2,fixed zerofill=true]{1.4931}  & \pgfmathprintnumber[fixed,precision=2,fixed zerofill=true]{1.79451}  & \pgfmathprintnumber[fixed,precision=2,fixed zerofill=true]{1.64489} \\ 
$k=2$ & \pgfmathprintnumber[fixed,precision=2,fixed zerofill=true]{2.93689}  & \pgfmathprintnumber[fixed,precision=2,fixed zerofill=true]{3.97416}  & \pgfmathprintnumber[fixed,precision=2,fixed zerofill=true]{2.9425}  & \pgfmathprintnumber[fixed,precision=2,fixed zerofill=true]{2.7795}  & \pgfmathprintnumber[fixed,precision=2,fixed zerofill=true]{3.63882}  & \pgfmathprintnumber[fixed,precision=2,fixed zerofill=true]{2.97164}  & \pgfmathprintnumber[fixed,precision=2,fixed zerofill=true]{2.49234}  & \pgfmathprintnumber[fixed,precision=2,fixed zerofill=true]{2.96154}  & \pgfmathprintnumber[fixed,precision=2,fixed zerofill=true]{2.83522} \\ 
$k=3$ & \pgfmathprintnumber[fixed,precision=2,fixed zerofill=true]{3.93472}  & \pgfmathprintnumber[fixed,precision=2,fixed zerofill=true]{4.94221}  & \pgfmathprintnumber[fixed,precision=2,fixed zerofill=true]{3.9753}  & \pgfmathprintnumber[fixed,precision=2,fixed zerofill=true]{3.74876}  & \pgfmathprintnumber[fixed,precision=2,fixed zerofill=true]{4.59222}  & \pgfmathprintnumber[fixed,precision=2,fixed zerofill=true]{3.95134}  & \pgfmathprintnumber[fixed,precision=2,fixed zerofill=true]{3.48604}  & \pgfmathprintnumber[fixed,precision=2,fixed zerofill=true]{3.96862}  & \pgfmathprintnumber[fixed,precision=2,fixed zerofill=true]{3.94311} \\ 
\bottomrule 
\end{tabular}

\end{figure}


\begin{figure}[h!]
    \centering
    \foreach \Reynolds in {0.01,1,10000} {
      \def\plotordera{1}
      \def\plotorderb{2}
      \def\plotorderc{3}
      \def\plotorderd{4}
      
      \ifnum100<\Reynolds
      \def\plotordera{0.5}
      \def\plotorderb{1.5}
      \def\plotorderc{2.5}
      \def\plotorderd{3.5}
      \fi
      \begin{minipage}{0.45\textwidth}
      \begin{tikzpicture}[scale=0.75]
        \begin{loglogaxis}[
            legend pos=outer north east,
            grid=major]
          \foreach \k in {0,1,2,3}{ \addplot table[x=msh,y=err_Gh] {pi6_tiltedhexagonal_err_Re\Reynolds_k\k.dat}; }
          \logLogSlopeTriangle{0.90}{0.31}{0.1}{\plotordera}{black};
          \logLogSlopeTriangle{0.90}{0.31}{0.1}{\plotorderb}{black};
          \logLogSlopeTriangle{0.90}{0.31}{0.1}{\plotorderc}{black};
          \logLogSlopeTriangle{0.90}{0.31}{0.1}{\plotorderd}{black};
          \legend{$k=0$,$k=1$,$k=2$,$k=3$}
        \end{loglogaxis}
      \end{tikzpicture}
      \subcaption{$\norm[\vU,h]{\uve}$ v. $h$ ($\Pe = \Reynolds$)}
    \end{minipage}
    \hspace{0.25cm}
    \begin{minipage}{0.45\textwidth}
      \begin{tikzpicture}[scale=0.75]
        \begin{loglogaxis}[
            legend pos=outer north east,
            grid=major]
          \foreach \k in {0,1,2,3}{ \addplot table[x=msh,y=err_Ph] {pi6_tiltedhexagonal_err_Re\Reynolds_k\k.dat}; }
          \logLogSlopeTriangle{0.90}{0.31}{0.1}{\plotordera}{black};
          \logLogSlopeTriangle{0.90}{0.31}{0.1}{\plotorderb}{black};
          \logLogSlopeTriangle{0.90}{0.31}{0.1}{\plotorderc}{black};
          \logLogSlopeTriangle{0.90}{0.31}{0.1}{\plotorderd}{black};
          \legend{$k=0$,$k=1$,$k=2$,$k=3$}
        \end{loglogaxis}      
      \end{tikzpicture}
      \subcaption{$\norm{\epsilon_h}$ v. $h$ ($\Pe = \Reynolds$)}
    \end{minipage}
    \bigskip\\
  } 
  \caption{Velocity error (left) and pressure error (right) versus meshsize $h$ on the hexagonal mesh sequence for $\Pe\in\{10^{-2},1,10^4\}$.\label{fig:err.hex}}
  \captionof{table}{Estimated asymptotic orders of convergence of the relative errors on the hexagonal mesh sequence.\label{tab:hex}}  
  \begin{tabular}{c|ccc|ccc|ccc}
\toprule
 & \multicolumn{3}{c|}{$\Pe = 0.01$} & \multicolumn{3}{c|}{$\Pe = 1$} & \multicolumn{3}{c}{$\Pe = 10000$} \\ 
\midrule 
& $\norm[\vU,h]{\uve}$ & $\norm{\ve[h]}$ & $\norm{\epsilon_h}$& $\norm[\vU,h]{\uve}$ & $\norm{\ve[h]}$ & $\norm{\epsilon_h}$& $\norm[\vU,h]{\uve}$ & $\norm{\ve[h]}$ & $\norm{\epsilon_h}$\\ \midrule 
$k=0$ & \pgfmathprintnumber[fixed,precision=2,fixed zerofill=true]{0.796138}  & \pgfmathprintnumber[fixed,precision=2,fixed zerofill=true]{1.42317}  & \pgfmathprintnumber[fixed,precision=2,fixed zerofill=true]{0.941369}  & \pgfmathprintnumber[fixed,precision=2,fixed zerofill=true]{0.601791}  & \pgfmathprintnumber[fixed,precision=2,fixed zerofill=true]{1.28781}  & \pgfmathprintnumber[fixed,precision=2,fixed zerofill=true]{0.847766}  & \pgfmathprintnumber[fixed,precision=2,fixed zerofill=true]{0.496851}  & \pgfmathprintnumber[fixed,precision=2,fixed zerofill=true]{0.689739}  & \pgfmathprintnumber[fixed,precision=2,fixed zerofill=true]{0.603763} \\ 
$k=1$ & \pgfmathprintnumber[fixed,precision=2,fixed zerofill=true]{1.74429}  & \pgfmathprintnumber[fixed,precision=2,fixed zerofill=true]{2.80765}  & \pgfmathprintnumber[fixed,precision=2,fixed zerofill=true]{2.20136}  & \pgfmathprintnumber[fixed,precision=2,fixed zerofill=true]{1.49401}  & \pgfmathprintnumber[fixed,precision=2,fixed zerofill=true]{2.50261}  & \pgfmathprintnumber[fixed,precision=2,fixed zerofill=true]{1.77352}  & \pgfmathprintnumber[fixed,precision=2,fixed zerofill=true]{1.49528}  & \pgfmathprintnumber[fixed,precision=2,fixed zerofill=true]{2.51692}  & \pgfmathprintnumber[fixed,precision=2,fixed zerofill=true]{2.41902} \\ 
$k=2$ & \pgfmathprintnumber[fixed,precision=2,fixed zerofill=true]{2.83832}  & \pgfmathprintnumber[fixed,precision=2,fixed zerofill=true]{3.88918}  & \pgfmathprintnumber[fixed,precision=2,fixed zerofill=true]{2.95583}  & \pgfmathprintnumber[fixed,precision=2,fixed zerofill=true]{2.45411}  & \pgfmathprintnumber[fixed,precision=2,fixed zerofill=true]{3.338}  & \pgfmathprintnumber[fixed,precision=2,fixed zerofill=true]{2.83864}  & \pgfmathprintnumber[fixed,precision=2,fixed zerofill=true]{2.50769}  & \pgfmathprintnumber[fixed,precision=2,fixed zerofill=true]{3.5859}  & \pgfmathprintnumber[fixed,precision=2,fixed zerofill=true]{4.15129} \\ 
$k=3$ & \pgfmathprintnumber[fixed,precision=2,fixed zerofill=true]{3.58828}  & \pgfmathprintnumber[fixed,precision=2,fixed zerofill=true]{4.56689}  & \pgfmathprintnumber[fixed,precision=2,fixed zerofill=true]{3.73594}  & \pgfmathprintnumber[fixed,precision=2,fixed zerofill=true]{3.36521}  & \pgfmathprintnumber[fixed,precision=2,fixed zerofill=true]{4.20097}  & \pgfmathprintnumber[fixed,precision=2,fixed zerofill=true]{3.52183}  & \pgfmathprintnumber[fixed,precision=2,fixed zerofill=true]{3.511}  & \pgfmathprintnumber[fixed,precision=2,fixed zerofill=true]{4.67034}  & \pgfmathprintnumber[fixed,precision=2,fixed zerofill=true]{4.43823} \\ 
\bottomrule 
\end{tabular}
  
\end{figure}

In Figures \ref{fig:err.tri} and \ref{fig:err.hex} we plot the errors $\norm[\vU,h]{\uve}$ and $\norm{\epsilon_h}$ estimated in Theorem \ref{thm:err.est} as functions of the meshsize $h$ for polynomial degrees $k$ ranging from 0 to 3.
In all the cases, the errors are normalized using the corresponding norm of the interpolant of the exact solution on a fine mesh with $k=3$.
We estimate the asymptotic convergence rates in Tables \ref{tab:tri} and \ref{tab:hex}, respectively, based on the following formula:
  \[
  \text{estimated convergence rate} \coloneq \frac{\log (\norm[\vU,h_2]{\uve[h_{2}]}) - \log (\norm[\vU,h_1]{\uve[h_1]})}{\log(h_{2}) - \log(h_1)  },
    \]
    where $h_1<h_2$ are the meshsizes corresponding to the last two mesh refinements.
    We also display the results for $\norm{\ve[h]}$, the $L^2$-norm of the error on the velocity.

From Figure \ref{fig:err.tri} and Table \ref{tab:tri}, we see that the estimated orders of convergence are almost perfectly matched for the energy norm of the velocity error $\norm[\vU,h]{\uve}$, with convergence in $h^{k+1}$ for $\Pe=0.01$, $h^{k+\frac12}$ for $\Pe= 10000$, and intermediate powers in between.
Similar considerations hold for the pressure error $\norm{\epsilon_h}$ for $\Pe=0.01$ and $\Pe=1$, whereas higher convergence rates than expected are observed for $\Pe=10000$.
This phenomenon will be further investigated in future works.
The $L^2$-norm of the velocity error, on the other hand, exhibits convergence in $h^{k+2}$, which corresponds to the classical supercloseness behaviour for HHO methods; see, e.g., \cite[Theorem 10]{Di-Pietro.Ern.ea:14}.
Similar considerations hold for the hexagonal mesh sequence (see Figure \ref{fig:err.hex} and Table \ref{tab:hex}) where, however, a slight degradation of the order of convergence for the energy norm of the velocity is observed already for the smallest value of the P\'eclet number.
In Figures \ref{fig:err.hex}a and \ref{fig:err.hex}c, it can be seen that the slope of the velocity error is still increasing in the last mesh refinement, which suggests that the asymptotic convergence rate has not been reached yet for $\Pe=0.01$ and $\Pe=1$.


\section{Proofs}\label{sec:proofs}

In this section we prove the main results stated in Section \ref{sec:discrete:main.results}.


\subsection{Well-posedness}\label{sec:proofs:well-posedness}
\begin{proof}[Proof of Theorem \ref{thm:well-posedness}]
  (i) \emph{Coercivity of $\mathrm{a}_h$.}
  Let $\uvv\in\UhD$.
  Writing \eqref{eq:ah.vel.reac:bis} for $\uvw=\uvv$ and using the definition \eqref{eq:vref.tref} of the reference time, we obtain for the advection-reaction norm
  $$
  \Ca\norm[\vel,\reac,h]{\uvv}^2\le\mathrm{a}_{\vel,\reac,h}(\uvv,\uvv).
  $$
  By definition \eqref{eq:norm.visc_vel.reac} of the viscous norm we have, on the other hand, $\norm[\visc,h]{\uvv}^2=\mathrm{a}_{\visc,h}(\uvv,\uvv)$.
  Observing that $\Ca\le 1$ and recalling the definition \eqref{eq:normUh} of the $\norm[\vU,h]{{\cdot}}$-norm, the conclusion follows.
  \medskip\\
  (ii) \emph{Inf-sup condition on $\mathrm{b}_h$.} Let $q_h\in\Ph\subset P$.
  From the surjectivity of the continuous divergence operator from $\vU$ to $P$, we infer the existence of $\vv[q_h]\in\vU$ such that $-\DIV\vv[q_h]=q_h$ and $\norm[H^1(\Omega)^d]{\vv}\lesssim\norm{q}$, with hidden constant only depending on $\Omega$.
  Using the fact above together with the definition \eqref{eq:lprojh} of the global $L^2$-orthogonal projector, the commuting property \eqref{eq:DT:commuting} of $\DT$, and the definition \eqref{eq:bh} of $\mathrm{b}_h$, we infer that
  \begin{equation}\label{eq:bh:stability:1}
    \norm{q_h}^2
    = -(\DIV\vv[q_h],q_h)
    = -(\lproj{k}(\DIV\vv[q_h]),q_h)
    = \mathrm{b}_h(\Ih\vv[q_h],q_h).
  \end{equation}
  Hence, denoting by $\$$ the supremum in the right-hand side of \eqref{eq:bh:stability}, we can write
  $$
  \norm{q_h}^2\le\$\norm[\vU,h]{\Ih\vv[q_h]}.
  $$
  The conclusion follows observing that
  \begin{equation}\label{eq:fortin.interpolator}
    \norm[\vU,h]{\Ih\vv[q_h]}
    \lesssim\Cb^{-1}\norm[1,h]{\Ih\vv[q_h]}
    \lesssim\Cb^{-1}\norm[H^1(\Omega)^d]{\vv[q_h]}
    \lesssim\Cb^{-1}\norm{q_h},
  \end{equation}
  where the first inequality follows from \eqref{eq:normUh:bnd} below, while for the second we have used the following continuity property for $\Ih$, which can be inferred from the continuity properties of $L^2$-projectors proved in \cite[Lemma 3.2]{Di-Pietro.Droniou:17}:
  For all $\vv\in\vU$, $\norm[1,h]{\Ih\vv}\lesssim\norm[H^1(\Omega)^d]{\vv}$.
  \medskip\\
  (iii) \emph{Continuity of $\mathrm{a}_h$.} Let $\uvw,\uvv\in\UhD$.
  The Cauchy--Schwarz inequality readily yields for the viscous bilinear form:
  \begin{equation}\label{eq:ah:continuity:1}
    |\mathrm{a}_{\visc,h}(\uvw,\uvv)|\le\norm[\visc,h]{\uvw}\norm[\visc,h]{\uvv}.
  \end{equation}

  To prove the boundedness of the advection-reaction bilinear form, we plug the expression \eqref{eq:aT.vel.reac} of $\mathrm{a}_{\vel,\reac,T}$ into the definition \eqref{eq:ah.vel.reac:sumT} of $\mathrm{a}_{\vel,\reac,h}$ and proceed to bound the three terms in the right-hand side.
  For any mesh element $T\in\Th$, recalling the definition \eqref{eq:Gb} of $\Gb$, and using H\"older and Cauchy--Schwarz inequalities, it is inferred that
  \begin{equation}\label{eq:ah:continuity:T1(T)}
    \begin{aligned}
      |(\vw[T],\Gb\uvv[T])_T|
      &\le\vref\norm[T]{\vw[T]}\norm[1,T]{\uvv[T]}
      \\
      &\lesssim\frac{\vref d_\Omega}{\visc}~\visc^{\frac12}\norm[T]{\vw[T]}~\visc^{\frac12}\norm[1,T]{\uvv[T]}
      \\
      &=\Pe_\Omega~\visc^{\frac12}\norm[T]{\vw[T]}~\visc^{\frac12}\norm[1,T]{\uvv[T]},
  \end{aligned}
  \end{equation}
  where we have used the fact that $d_\Omega^{-1}\lesssim 1$ to pass to the second line and the definition \eqref{eq:Peh.trefh} of $\Pe_\Omega$ to conclude.
  We next recall the following discrete Poincar\'e inequality for HHO spaces proved in \cite[Proposition 5.4]{Di-Pietro.Droniou:17}:
  For all $\uvw\in\UhD$,
  \begin{equation}\label{eq:poincare}
    \norm{\vw[h]}\lesssim\norm[1,h]{\uvw},
  \end{equation}
  where the $\norm[1,h]{{\cdot}}$-norm is defined by \eqref{eq:ah.visc:stability}.
  Using \eqref{eq:ah:continuity:T1(T)} followed by a discrete Cauchy--Schwarz inequality on the sum over $T\in\Th$, \eqref{eq:poincare}, and the global norm equivalence \eqref{eq:ah.visc:stability} yields
  \begin{equation}\label{eq:ah:continuity:T1}
    \left|\sum_{T\in\Th}(\vw[T],\Gb\uvv[T])_T\right|
    \lesssim\Pe_\Omega~\visc^{\frac12}\norm{\vw}~\visc^{\frac12}\norm[1,h]{\uvv[T]}
    \lesssim\Pe_\Omega\norm[\visc,h]{\uvw}\norm[\visc,h]{\uvv}.
  \end{equation}
  For the second term, the Cauchy--Schwarz inequality followed by the definition \eqref{eq:vref.tref} of the reference time $\tref$ give
  \begin{equation}\label{eq:ah:continuity:T2}
    \left|\sum_{T\in\Th}\reac(\vw[T],\vv[T])_T\right|
    \le\left(
    \sum_{T\in\Th}\tref^{-1}\norm[T]{\vw[T]}^2
    \right)^{\frac12}\left(
    \sum_{T\in\Th}\tref^{-1}\norm[T]{\vv[T]}^2
    \right)^{\frac12}.
  \end{equation}
  Finally, using again the Cauchy--Schwarz inequality together with the fact that $\svel^-\le|\svel|$, we have for the third term
  \begin{equation}\label{eq:ah:continuity:T3}
    \left|\sum_{T\in\Th}\mathrm{s}^-_{\vel,T}(\uvw[T],\uvv[T])\right|
    \lesssim\left(
    \frac12\sum_{T\in\Th}\sum_{F\in\Fh[T]}\norm[F]{|\svel|^{\frac12}(\vw[F]-\vw[T])}^2
    \right)^{\frac12}\left(
    \frac12\sum_{T\in\Th}\sum_{F\in\Fh[T]}\norm[F]{|\svel|^{\frac12}(\vv[F]-\vv[T])}^2
    \right)^{\frac12}.
  \end{equation}
  Combining \eqref{eq:ah:continuity:T1}--\eqref{eq:ah:continuity:T3}, and recalling the definition \eqref{eq:norm.visc_vel.reac} of the viscous and advection-reaction norm, we conclude that
  \begin{equation}\label{eq:ah:continuity:2}
    |\mathrm{a}_{\vel,\reac,h}(\uvw,\uvv)|
    \lesssim
    \Pe_\Omega\norm[\visc,h]{\uvw}\norm[\visc,h]{\uvv}
    + \norm[\vel,\reac,h]{\uvw}\norm[\vel,\reac,h]{\uvv}
    \le (1+\Pe_\Omega)\norm[\vU,h]{\uvw}\norm[\vU,h]{\uvv}.
  \end{equation}
  Observing that
  $$
  |\mathrm{a}_h(\uvw,\uvv)|\le|\mathrm{a}_{\visc,h}(\uvw,\uvv)| + |\mathrm{a}_{\vel,\reac,h}(\uvw,\uvv)|
  $$
  and using \eqref{eq:ah:continuity:1} and \eqref{eq:ah:continuity:2} to bound the terms in the right-hand side, \eqref{eq:ah:continuity} follows.
  \medskip\\
  (iv) \emph{Well-posedness and a priori bounds.}
  Denote by $\mathfrak{f}$ the linear functional on $\UhD$ such that $\langle\mathfrak{f},\uvv\rangle=(\vf,\vv[h])$ for all $\uvv\in\UhD$.
  The well-posedness of problem \eqref{eq:discrete} with a priori bounds as in \eqref{eq:a-priori bounds} but with $\visc^{-\frac12}\norm{\vf}$ replaced by $\norm[\vU^*,h]{\mathfrak{f}}$ follows from an application of \cite[Theorem 2.34]{Ern.Guermond:04} after observing that the second condition in Eq. (2.28) therein is a consequence of the first in a finite-dimensional setting; see also \cite[Theorem 3.4.5]{Boffi.Brezzi.ea:13} for the corresponding algebraic result.
  
  The estimate $\norm[\vU^*,h]{\mathfrak{f}}\lesssim\visc^{-\frac12}\norm{\vf}$ that allows one to write \eqref{eq:a-priori bounds} is proved bounding the argument of the supremum in the definition \eqref{eq:normUh.dual} of the dual norm as follows:
  $$
  \langle\mathfrak{f},\uvv\rangle
  = (\vf,\vv[h])
  \le\norm{\vf}~\norm{\vv[h]}
  \lesssim\norm{\vf}~\norm[1,h]{\uvv}
  \lesssim\norm{\vf}~\visc^{-\frac12}\norm[\visc,h]{\uvv}
  \lesssim\norm{\vf}~\visc^{-\frac12}\norm[\vU,h]{\uvv},
  $$
  where we have used the Cauchy--Schwarz inequality in the first bound, the discrete Poincar\'e inequality \eqref{eq:poincare} in the second bound, the norm equivalence \eqref{eq:ah.visc:stability} in the third bound, and the definition \eqref{eq:normUh} of the $\norm[\vU,h]{{\cdot}}$-norm to conclude.
\end{proof}
The following proposition was used in point (ii) of the above proof.
\begin{proposition}[Equivalence of global norms]
  For all $\uvv\in\UhD$ it holds with $\Cb$ as in \eqref{eq:bh:stability}:
  \begin{equation}\label{eq:normUh:bnd}
    \visc^{\frac12}\norm[1,h]{\uvv}\lesssim\norm[\vU,h]{\uvv}\lesssim\Cb^{-1}\norm[1,h]{\uvv}.
  \end{equation}
\end{proposition}
\begin{proof}
  Let $\uvv = ((\vv[T])_{T\in\Th},(\vv[F])_{F\in\Fh}\big)\in\UhD$.
  To prove the first inequality in \eqref{eq:normUh:bnd}, it suffices to use \eqref{eq:ah.visc:stability} followed by \eqref{eq:normUh} to infer
  $$
  \visc^{\frac12}\norm[1,h]{\uvv}\lesssim\norm[\visc,h]{\uvv}\le\norm[\vU,h]{\uvv}.
  $$

  To prove the second inequality in \eqref{eq:normUh:bnd}, we estimate the terms that compose the $\norm[\vU,h]{{\cdot}}$-norm; see \eqref{eq:normUh}.
  We start by observing that, using again \eqref{eq:ah.visc:stability}, it holds
  \begin{equation}\label{eq:normUh.bnd:1}
    \norm[\visc,h]{\uvv}^2\lesssim\visc\norm[1,h]{\uvv}^2.
  \end{equation}
  Let us bound the second term in the right-hand side of \eqref{eq:normUh}.
  By definition \eqref{eq:PeT} of the local P\'eclet number $\Pe_T$, it is readily inferred for all $T\in\Th$ that
  $$
  \frac12\sum_{F\in\Fh[T]}\norm[F]{|\svel|^{\frac12}(\vv[F]-\vv[T])}^2
  \le\frac12\visc\Pe_T\sum_{F\in\Fh[T]}h_F^{-1}\norm[F]{\vv[F]-\vv[T]}^2
  \lesssim\visc\Pe_T\norm[1,T]{\uvv[T]}^2.
  $$  
  Summing over $T\in\Th$ and recalling \eqref{eq:ah.visc:stability}, we conclude that
  $$
  \frac12\sum_{T\in\Th}\sum_{F\in\Fh[T]}\norm[F]{|\svel|^{\frac12}(\vv[F]-\vv[T])}^2
  \lesssim\visc\Pe_h\norm[1,h]{\vv[T]}^2.
  $$
  On the other hand, using the definition \eqref{eq:Peh.trefh} of the global reference time together with the Poincar\'{e} inequality for HHO spaces proved in \cite[Proposition~5.4]{Di-Pietro.Droniou:17} yields
  $$
  \sum_{T\in\Th}\tref^{-1}\norm[T]{\vv[T]}^2
  \lesssim\tref[h]^{-1}\norm[1,h]{\uvv}^2.
  $$
  From the above relations, we get the following bound for the second term in the right-hand side of \eqref{eq:normUh}:
  \begin{equation}\label{eq:normUh.bnd:2}
    \norm[\vel,\reac,h]{\uvv}^2\lesssim \left(\visc\Pe_h + \tref[h]^{-1}\right)\norm[1,h]{\uvv}^2.
  \end{equation}
  The second inequality in \eqref{eq:normUh:bnd} then follows using \eqref{eq:normUh.bnd:1} and \eqref{eq:normUh.bnd:2} to bound the right-hand side of \eqref{eq:normUh}. 
\end{proof}


\subsection{Convergence}\label{sec:proofs:convergence}

This section contains the proof of Theorem \ref{thm:err.est} preceeded by the required preliminary results.

\subsubsection{Preliminary results}

In this section we prove three lemmas that contain consistency results for the bilinear forms appearing in \eqref{eq:discrete}.

\begin{lemma}[Consistency of the viscous bilinear form]\label{lem:ah.visc:consistency}
  For any $\vw\in H_0^1(\Omega)^d\cap H^{k+2}(\Th)^d$ such that $\LAPL\vw\in L^2(\Omega)^d$ and all $\uvv\in\UhD$, it holds that
  \begin{equation}\label{eq:ah.visc:consistency}
    \Eh[\mathrm{a},\visc,h](\vw;\uvv)\coloneq
    \left|
    \visc(\LAPL\vw,\vv[h]) + \mathrm{a}_{\visc,h}(\Ih\vw,\uvv)
    \right|\lesssim \left(
    \sum_{T\in\Th}\visc h_T^{2(k+1)}\seminorm[H^{k+2}(T)^d]{\vw}^2
    \right)^{\frac12}\norm[\visc,h]{\uvv}.
  \end{equation}
\end{lemma}
\begin{proof}
  In the proof we set, for the sake of brevity, $\cvw\coloneq\rT\IT\vw=\veproj{k+1}\vw$ (see \eqref{eq:rTIT=veproj}).
  Integrating by parts element by element, it is inferred that
  \begin{equation}\label{eq:ah.visc:stability:1}
    \visc(\LAPL\vw,\vv[h])
    =-\sum_{T\in\Th}\left(
      \visc(\GRAD\vw,\GRAD\vv[T])_T + \sum_{F\in\Fh[T]}\visc(\GRAD\vw\normal_{TF},\vv[F]-\vv[T])_F
      \right),
  \end{equation}
  where we have used the fact that $\GRAD\vw$ has continuous normal trace across any $F\in\Fhi$ (cf., e.g., \cite[Lemma 1.24]{Di-Pietro.Ern:12}) and that $\vv[F]=\vec{0}$ for all $F\in\Fhb$ to insert $\vv[F]$ into the second term.
  On the other hand, expanding first $\mathrm{a}_{\visc,h}$ then $\mathrm{a}_{\visc,T}$ according to their respective definitions \eqref{eq:ah.visc.sumT} and \eqref{eq:aT.visc}, and using for any $T\in\Th$ the definition \eqref{eq:rT} of $\rT\uvv[T]$ with $\vw=\cvw$, we arrive at
  \begin{equation}\label{eq:ah.visc:stability:2}
    \mathrm{a}_{\visc,h}(\Ih\vw,\uvv)
    = \sum_{T\in\Th}\left(
      \visc(\GRAD\cvw,\GRAD\vv[T])_T + \sum_{F\in\Fh[T]}\visc(\GRAD\cvw\normal_{TF},\vv[F]-\vv[T])_F
      \right)
    + \sum_{T\in\Th}\mathrm{s}_{\visc,T}(\IT\vw,\uvv[T]).
  \end{equation}
  Summing \eqref{eq:ah.visc:stability:1} and \eqref{eq:ah.visc:stability:2}, observing that the first terms inside parentheses cancel out by definition \eqref{eq:eproj} of $\veproj[T]{k+1}$ since $\vv[T]\in\Poly{k}(T)^d\subset\Poly{k+1}(T)^d$, and using Cauchy--Schwarz inequalities, we infer that
  \begin{equation}\label{eq:ah.visc:stability:3}
    \begin{aligned}
      \Eh[\mathrm{a},\visc,h](\vw;\uvv)
      &\le
      \left( \sum_{T\in\Th}\visc h_T\norm[\partial T]{\GRAD(\cvw-\vw)\normal_T}^2\right)^{\frac12}\left(
        \visc\sum_{T\in\Th}\sum_{F\in\Fh[T]}h_F^{-1}\norm[F]{\vv[F]-\vv[T]}^2
        \right)^{\frac12}
      \\
      &\quad +
      \left(
      \sum_{T\in\Th}\mathrm{s}_{\visc,T}(\IT\vw,\IT\vw)
      \right)^{\frac12}\left(
      \sum_{T\in\Th}\mathrm{s}_{\visc,T}(\uvv[T],\uvv[T])
      \right)^{\frac12}
      \eqcolon\term_1 + \term_2.
    \end{aligned}
  \end{equation}
  Using the optimal approximation properties \eqref{eq:approx.trace} of the elliptic projector with $\xi=1$, $l=k+1$, $s=k+2$, and $m=1$ together with the norm equivalence \eqref{eq:ah.visc:stability}, it is readily inferred that
  $$
  \term_1\lesssim\left( \sum_{T\in\Th} \visc h_T^{2(k+1)}\seminorm[H^{k+2}(T)^d]{\vw}^2\right)^{\frac12}\norm[\visc,h]{\uvv}.
  $$
  On the other hand, recalling the approximation properties \eqref{eq:sT.visc:consistency} of the stabilization bilinear form and the definition \eqref{eq:ah.visc:stability} of $\norm[\visc,h]{{\cdot}}$, we get
  $$
  \term_2\lesssim\left( \sum_{T\in\Th} \visc h_T^{2(k+1)}\seminorm[H^{k+2}(T)^d]{\vw}^2\right)^{\frac12}\norm[\visc,h]{\uvv}.
  $$
  Plugging the above estimates into \eqref{eq:ah.visc:stability:3}, \eqref{eq:ah.visc:consistency} follows.
\end{proof}

\begin{lemma}[Consistency of the advection-reaction bilinear form]\label{lem:ah.vel.reac:consistency}
  For all $\vw\in H_0^1(\Omega)^d\cap H^{k+1}(\Th)^d$ and all $\uvv\in\UhD$, it holds that
  \begin{equation}\label{eq:ah.vel.reac:consistency}
    \begin{aligned}
      \Eh[\mathrm{a},\vel,\reac,h](\vw;\uvv)
      &\coloneq\left|
      ((\vel\SCAL\GRAD)\vw + \reac\vec{w},\vv[h]) - \mathrm{a}_{\vel,\reac,h}(\Ih\vw,\uvv)
      \right|
      \\
      &\lesssim\left(
      \sum_{T\in\Th}\vref\min(1,\Pe_T)h_T^{2k+1}\seminorm[H^{k+1}(T)^d]{\vw}^2
      \right)^{\frac12}\norm[\vU,h]{\uvv}.
    \end{aligned}
  \end{equation}
\end{lemma}
\begin{proof}
  Integrating by parts the first term inside absolute value in \eqref{eq:ah.vel.reac:consistency}, recalling that $\DIV\vel=0$ by assumption, and adding the following quantity (recall that $\vw$ has continuous trace across interfaces and that $\vv[F]=\vec{0}$ for all $F\in\Fhb$):
  $$
  -\sum_{T\in\Th}\sum_{F\in\Fh[T]} (\svel \vw,\vv[F])_F=0,
  $$
  we have, expanding the definition \eqref{eq:Gb} of the discrete advective derivative and of the upwind stabilization,
  \begin{align*}
    \Eh[\mathrm{a},\vel,\reac,h](\vw;\uvv)    
    &=
    \cancel{\sum_{T\in\Th} (\vw-\vhw[T], \reac \vv[T])_T}
    +
    \underbrace{%
      \sum_{T\in\Th} (\vhw[T]-\vw[|T], (\vel\SCAL\GRAD)\vv[T])_T%
    }_{\term_1}
    \\
    &\qquad + \underbrace{%
      \sum_{T\in\Th}\sum_{F\in\Fh[T]}(\svel (\vhw[T]-\vw[|T]), \vv[F] - \vv[T])_F
      - \sum_{T\in\Th}\sum_{F\in\Fh[T]}(\svel^-(\vhw[F]-\vhw[T]),\vv[F]-\vv[T])_F
    }_{\term_2},
  \end{align*}
  where we have used the definition \eqref{eq:lproj} of the orthogonal projector (recall that $\vhw[T]=\vlproj[T]{k}\vw$) together with the fact that $(\reac\vv[T])\in\Poly{k}(T)^d$ since $\reac$ is constant over $\Omega$ by assumption to cancel the first term in the right-hand side.
  
   Observing that  $(\lproj[T]{0}\vel)\SCAL\GRAD \vv[T]\in\Poly{k-1}(T)^d\subset\Poly{k}(T)^d$, we can use again the definition \eqref{eq:lproj} of the $L^2$-orthogonal projector to write
  \[
  \term_1
  =\sum_{T\in\Th} (\vhw[T]-\vw, (\vel-\lproj[T]{0}\vel)\SCAL \GRAD \vv[T]).
  \]
  Using the H\"older and Cauchy--Schwarz inequalities, we can now estimate the first term as follows:
  \begin{equation}\label{eq:ah.vel.reac:consistency:T1}
    \begin{aligned}
      |\term_1| &\le
      \sum_{T\in\Th} \norm[L^\infty(T)^d]{\vel - \lproj[T]{0}\vel}\norm[T]{\vhw[T]-\vw}\norm[T]{\GRAD \vv[T]}
      \\
      & \lesssim \sum_{T\in\Th}\tref^{-\frac12} h_T^{k+1}\seminorm[H^{k+1}(T)^d]{\vw}~\tref^{-\frac12}\norm[T]{\vv[T]} \\
      & \le \left(
      \sum_{T\in\Th} \tref^{-1} h_T^{2(k+1)}\seminorm[H^{k+1}(T)^d]{\vw}^2
      \right)^{\frac12} \norm[\vel,\reac,h]{\uvv[h]}.
    \end{aligned}
  \end{equation}
  To pass to the second line, we have used the Lipschitz continuity of $\vel$ together with the definition \eqref{eq:vref.tref} of the reference time $\tref$ to write for the first factor $\norm[L^\infty(T)^d]{\vel - \lproj[T]{0}\vel}\le L_{\vel,T}h_T\le\tref^{-1}h_T$, the approximation properties \eqref{eq:approx} of $\vlproj[T]{k}$ with $\xi=0$, $l=k$, $m=0$, and $s=k+1$ to bound the second factor, and the inverse inequality \eqref{eq:trace.inv} to bound the third.
  The inequality in the third line is an immediate consequence of the discrete Cauchy--Schwarz inequality.

  The term $\term_2$ is estimated using the following decomposition based on the local P\'eclet number:
  $$
  \term_2 = \term_2^{\rm d} + \term_2^{\rm a},
  $$
  where the subscript ``d'' (for ``diffusion-controlled'') corresponds to integrals where $|\Pe_{TF}|<1$, while the subscript ``a'' (for ``advection-controlled'') to integrals where $|\Pe_{TF}|\ge 1$.
  Henceforth, we denote by $\chi_{|\Pe_{TF}|< 1}$ and $\chi_{|\Pe_{TF}|\ge 1}$ the two characteristic functions of the corresponding regions.
  The linearity and idempotency of $\vlproj[F]{k}$ followed by its $L^2(F)^d$-continuity yield
  $$
  \norm[F]{\vhw[F]-\vhw[T]}
  = \norm[F]{\vlproj[F]{k}(\vw-\vhw[T])}
  \le\norm[F]{\vw-\vhw[T]}.
  $$
  Hence we can write for the diffusion-controlled contribution, using the H\"older and Cauchy--Schwarz inequalities,
  \begin{equation}\label{eq:ah:consistency:T2d+T3d}
    \begin{aligned}
      \term_2^{\rm d}
      &\lesssim\sum_{T\in\Th}\sum_{F\in\Fh[T]}
      \norm[L^\infty(F)]{\svel\chi_{|\Pe_{TF}|< 1}}\left(
      \norm[F]{\vhw[T]-\vw} + \norm[F]{\vhw[T]-\vhw[F]}
      \right)\norm[F]{\vv[F]-\vv[T]}
      \\
    &\lesssim\sum_{T\in\Th}\sum_{F\in\Fh[T]}
    \vref^{\frac12}\norm[L^\infty(F)]{\Pe_{TF}\chi_{|\Pe_{TF}|< 1}}^{\frac12}\norm[F]{\vw-\vhw[T]}~\left(\frac{\visc}{h_F}\right)^{\frac12}\norm[F]{\vv[F]-\vv[T]}
    \\
    &\lesssim\left(
    \sum_{T\in\Th}\vref\min(1,\Pe_T)\norm[\partial T]{\vw-\vhw[T]}^2
    \right)^{\frac12}\norm[\visc,h]{\uvv}.
  \end{aligned}
  \end{equation}
  To pass to the second line, we have multiplied and divided by $(\visc/h_F)^{\frac12}\simeq(h_T/\visc)^{-\frac12}$, recalled the definition \eqref{eq:PeTF} of the local face P\'eclet number to write $(h_T/\visc)^{-\frac12}\norm[L^\infty(F)]{\svel\chi_{|\Pe_{TF}|<1}}^{\frac12}=\norm[L^\infty(F)]{\Pe_{TF}\chi_{|\Pe_{TF}|< 1}}^{\frac12}$, and estimated $\norm[L^\infty(F)]{\svel\chi_{|\Pe_{TF}|< 1}}^{\frac12}\le\norm[L^\infty(F)]{\svel}^{\frac12}\le\vref^{\frac12}$.
  To pass to the third line, we have used a discrete Cauchy--Schwarz inequality together with the definition \eqref{eq:norm.visc_vel.reac} of the $\norm[\visc,h]{{\cdot}}$-norm.
  
  For the advection-controlled contribution, using again the H\"older and Cauchy--Schwarz inequalities we have, on the other hand,
  \begin{equation}\label{eq:ah:consistency:T2a+T3a}
    \begin{aligned}
      \term_2^{\rm a}
      &\le\left(
      \sum_{T\in\Th}\sum_{F\in\Fh[T]}\norm[L^\infty(F)]{\svel\chi_{|\Pe_{TF}|\ge 1}} \norm[F]{\vw-\vhw[T]}^2
      \right)^{\frac12}
      \\
      &\qquad\times\left(
      \sum_{T\in\Th}\sum_{F\in\Fh[T]}\left(|\svel|\chi_{|\Pe_{TF}|\ge 1}(\vv[F]-\vv[T]),\vv[F]-\vv[T]\right)_F
      \right)^{\frac12}
      \\
      &\lesssim\left(
      \sum_{T\in\Th}\vref\min(1,\Pe_T)\norm[\partial T]{\vw-\vhw[T]}^2
      \right)^{\frac12}\norm[\vel,\reac,h]{\uvv}.
    \end{aligned}
  \end{equation}
  Owing to the approximation properties~\eqref{eq:approx.trace} of $\vhw[T]=\vlproj[T]{k}\vw$ it holds, for all $T\in\Th$ and all $F\in\Fh[T]$,
  $$
  \norm[F]{\vw-\vhw[T]}^2 \lesssim h_T^{k+\frac12} \seminorm[H^{k+1}(T)^d]{\vw}.
  $$
  Plugging this bound into \eqref{eq:ah:consistency:T2d+T3d} and \eqref{eq:ah:consistency:T2a+T3a}, we conclude that
  \begin{equation}\label{eq:ah.vel.reac:consistency:T2+T3}
    |\term_2|\lesssim
    \left(
    \sum_{T\in\Th} \vref\min(1,\Pe_{T}) h_T^{2k+1}\seminorm[H^{k+1}(T)^d]{\vw}^2
    \right)^{\frac12} \norm[\vU,h]{\uvv[h]}.    
  \end{equation}
  Combining \eqref{eq:ah.vel.reac:consistency:T1} and \eqref{eq:ah.vel.reac:consistency:T2+T3}, \eqref{eq:ah.vel.reac:consistency} follows.
\end{proof}
\begin{lemma}[Consistency of the velocity-pressure coupling bilinear form]\label{lem:bh:consistency}
   For any $q\in P\cap H^1(\Omega)\cap H^{k+1}(\Th)$ and all $\uvv\in\UhD$, it holds that
    \begin{equation}\label{eq:bh:consistency}
      \Eh[\mathrm{b},h](q;\uvv)\coloneq
      \left|
      -(\vv[h],\GRAD q) + \mathrm{b}_h(\uvv,\lproj{k}q)
      \right|\lesssim \left(
      \sum_{T\in\Th}\visc^{-1}h_T^{2(k+1)}\seminorm[H^{k+1}(T)]{q}^2
      \right)^{\frac12}\norm[\visc,h]{\uvv}.
    \end{equation}
\end{lemma}
\begin{proof}
    Expanding $\mathrm{b}_h$ then $\DT$ according to their respective definitions \eqref{eq:bh} and \eqref{eq:DT}, we obtain
  \begin{equation}\label{eq:bh:consistency:1}
    \begin{aligned}
      \mathrm{b}_h(\uvv,\lproj{k}q)
      &= -\sum_{T\in\Th}\left(
        -(\vv[T],\GRAD\lproj[T]{k}q)_T + \sum_{F\in\Fh[T]}(\vv[F]\SCAL\normal_{TF},\lproj[T]{k}q)_F
        \right)
      \\
      &= -\sum_{T\in\Th}\left(
        (\DIV\vv[T],q)_T + \sum_{F\in\Fh[T]}((\vv[F]-\vv[T])\SCAL\normal_{TF},\lproj[T]{k}q)_F
        \right),
    \end{aligned}
  \end{equation}
  where, to pass to the second line, we have integrated by parts the first term inside parentheses and we have used the fact that $\DIV\vv[T]\in\Poly{k-1}(T)\subset\Poly{k}(T)$ and the definition \eqref{eq:lproj} of the $L^2$-orthogonal projector to write $q$ instead of $\lproj[T]{k}q$ in the first term.
  On the other hand, an element by element integration by parts gives
  \begin{equation}\label{eq:bh:consistency:2}
    -(\vv[h],\GRAD q)
    = \sum_{T\in\Th}\left(
      (\DIV\vv[T],q)_T + \sum_{F\in\Fh[T]}((\vv[F]-\vv[T])\SCAL\normal_{TF},q)_F
      \right),
  \end{equation}
  where, to insert $\vv[F]$ into the second term, we have used the fact that the jumps of $q$ vanish across interfaces (a consequence of the regularity assumption $q\in H^1(\Omega)$, see \cite[Lemma 1.23]{Di-Pietro.Ern:12}) together with the fact that $\vv[F]=\vec{0}$ for all $F\in\Fhb$.
  Summing \eqref{eq:bh:consistency:1} and \eqref{eq:bh:consistency:2}, taking absolute values, and using the Cauchy--Schwarz inequality to bound the right-hand side of the resulting expression, it is inferred that
  $$
  \begin{aligned}
    \Eh[\mathrm{b},h](q;\uvv)
    &\le\left(
    \sum_{T\in\Th}\sum_{F\in\Fh[T]}\visc^{-1}h_F\norm[F]{\lproj[T]{k}q-q}^2
    \right)^{\frac12}\left(
    \visc\sum_{T\in\Th}\sum_{F\in\Fh[T]}h_F^{-1}\norm[F]{\vv[F]-\vv[T]}^2
    \right)^{\frac12}
    \\
    &\lesssim\left(
    \sum_{T\in\Th}\visc^{-1}h_T^{k+1}\seminorm[H^{k+1}(T)]{q}^2
    \right)^{\frac12}\norm[\visc,h]{\uvv},
  \end{aligned}
  $$
  where we have used the optimal approximation properties \eqref{eq:approx.trace} of $\lproj[T]{k}$ with $\xi=0$, $l=k$, $s=k+1$, and $m=0$ together with the definition \eqref{eq:ah.visc:stability} of the $\norm[1,h]{{\cdot}}$ norm and the norm equivalence \eqref{eq:ah.visc:stability} to conclude.
\end{proof}

\subsubsection{Error estimates and convergence}

We are now ready to prove Theorem \ref{thm:err.est}.

\begin{proof}[Proof of Theorem \ref{thm:err.est}]
  (i) \emph{Error estimates.} The error estimates \eqref{eq:err.est} are a consequence of \cite[Theorem 2.34]{Ern.Guermond:04} applied to the error equation \eqref{eq:discrete.error}; see also the discussion in point (iv) of the proof of Theorem \ref{thm:well-posedness} in Section \ref{sec:proofs:well-posedness}.
  \medskip\\
  (ii) \emph{Convergence rate.} Let $\uvv\in\UhD$.
  Using the definition \eqref{eq:frakR} of $\mathfrak{R}(\vu,p)$ together with the fact that \eqref{eq:strong:momentum} is satisfied almost everywhere in $\Omega$ by the weak solution $(\vu,p)$ of \eqref{eq:weak}, it is inferred for all $\uvv\in\UhD$
  $$
  \begin{aligned}
    \langle\mathfrak{R}(\vu,p),\uvv\rangle
    &= -\nu(\LAPL\vu,\vv[h])-\mathrm{a}_{\visc,h}(\uvhu,\uvv)
    \\
    &\quad + ((\vel\SCAL\GRAD)\vu+\reac\vu,\vv[h]) - \mathrm{a}_{\vel,\reac,h}(\uvhu,\uvv)
    \\
    &\quad + (\GRAD p,\vv[h]) - \mathrm{b}_h(\uvv,\hp).
  \end{aligned}
  $$
  Hence, passing to absolute values and using the triangle inequality, we can write
  \begin{equation}\label{eq:err.est:conv.rate:1}
    \left|\langle\mathfrak{R}(\vu,p),\uvv\rangle\right|
    \le \Eh[\mathrm{a},\visc,h](\vu;\uvv) + \Eh[\mathrm{a},\vel,\reac,h](\vu;\uvv) + \Eh[\mathrm{b},h](p;\uvv),
  \end{equation}
  with error contributions respectively defined in Lemmas \ref{lem:ah.visc:consistency}, \ref{lem:ah.vel.reac:consistency}, and \ref{lem:bh:consistency}.
  Using \eqref{eq:ah.visc:consistency}, \eqref{eq:ah.vel.reac:consistency}, and \eqref{eq:bh:consistency}, respectively, to bound the terms in the right-hand side of \eqref{eq:err.est:conv.rate:1}, it is readily inferred that
  \begin{equation}\label{eq:err.est:conv.rate:2}
    \left|\langle\mathfrak{R}(\vu,p),\uvv\rangle\right|
    \le
    \left[\sum_{T\in\Th}\left( h_T^{2(k+1)}\cN[1,T] + \min(1,\Pe_T)h_T^{2k+1}\cN[2,T] \right) \right]^{\frac12}\norm[\vU,h]{\uvv}.
  \end{equation}
  Expanding $\norm[\vU^*,h]{\mathfrak{R}(\vu,p)}$ according to its definition \eqref{eq:normUh.dual} and using \eqref{eq:err.est:conv.rate:2}, \eqref{eq:conv.rate} follows.
\end{proof}


\appendix

\section{Flux formulation}\label{sec:flux.formulation}

In this section we reformulate the discrete problem in terms of numerical fluxes, and show that local momentum and mass balances hold.
Let a mesh element $T\in\Th$ be fixed, and define the boundary difference space
$$
\underline{\vec{D}}_{\partial T}^k\coloneq\left\{
\underline{\vec{\alpha}}_{\partial T} = (\vec{\alpha}_F)_{F\in\Fh[T]}\ST\vec{\alpha}_F\in\Poly{k}(F)^d\mbox{ for all }F\in\Fh[T]
\right\}.
$$
We introduce the boundary difference operator $\uDpT:\UT\to\underline{\vec{D}}_{\partial T}^k$ such that, for all $\uvv[T]\in\UT$,
$$
\uDpT\uvv[T]\coloneq(\vv[F]-\vv[T|F])_{F\in\Fh[T]}.
$$
The following result was proved in the scalar case in \cite[Proposition 3]{Di-Pietro.Tittarelli:17}.
\begin{proposition}[Reformulation of the viscous stabilization bilinear form]\label{prop:sT.visc:bis}
  Let an element $T\in\Th$ be fixed, and let $\{\mathrm{s}_{\visc,T}\ST T\in\Th\}$ denote a family of viscous stabilization bilinear forms that satisfy assumptions (S1)--(S3) in Remark \ref{rem:sT.visc}, and which depend on their arguments only via the difference operators defined by \eqref{eq:dT.dTF}.
  Then, for all $T\in\Th$ and all $\uvw[T],\uvv[T]\in\UT$ it holds that
  \begin{equation}\label{eq:sT.visc:bis}
    \mathrm{s}_{\visc,T}(\uvw[T],\uvv[T])
    = \mathrm{s}_{\visc,T}(\uvw[T],(\vec{0},\uDpT\uvv[T])).
  \end{equation}
\end{proposition}
The reformulation \eqref{eq:sT.visc:bis} of the viscous stabilization term prompts the following definition:
For all $T\in\Th$, the boundary residual operator $\underline{\vec{R}}_{\partial T}^k:\UT\to\underline{\vec{D}}_{\partial T}^k$ is such that, for all $\uvw[T]\in\UT$,
$$
\underline{\vec{R}}_{\partial T}^k\uvw[T]=(\vec{R}_{TF}^k\uvw[T])_{F\in\Fh[T]}
$$
satisfies
\begin{equation}\label{eq:RpT}
  -\sum_{F\in\Fh[T]}(\vec{R}_{TF}^k\uvw[T],\vec{\alpha}_F)_F
  = \mathrm{s}_{\visc,T}(\uvw[T],(\vec{0},\underline{\vec{\alpha}}_{\partial T}))\qquad
   \forall\underline{\vec{\alpha}}_{\partial T}\in\underline{\vec{D}}_{\partial T}^k.
\end{equation}
\begin{theorem}[Flux formulation]\label{thm:flux.formulation}
  Under the assumptions of Proposition \ref{prop:sT.visc:bis}, denote by $(\uvu,p_h)\in\UhD\times\Ph$ the unique solution of problem \eqref{eq:discrete} and, for all $T\in\Th$ and all $F\in\Fh[T]$, define the numerical normal trace of the momentum flux as
  $$
  \vec{\Phi}_{TF} \coloneq \vec{\Phi}_{TF}^{\rm cons} + \vec{\Phi}_{TF}^{\rm stab}
  $$
  with consistency and stabilization contributions given by, respectively,
  $$
  \vec{\Phi}_{TF}^{\rm cons} \coloneq -\visc\GRAD(\rT\uvu[T])\normal_{TF} + \svel\vu[T] + p_T\normal_{TF},\qquad
  \vec{\Phi}_{TF}^{\rm stab} \coloneq \vec{R}_{TF}^k\uvu[T] + \svel^-(\vu[T]-\vu[F]).
  $$
  Then, for all $T\in\Th$ the following local balances hold:
  For all $\vv[T]\in\Poly{k}(T)^d$ and all $q_T\in\Poly{k}(T)$,
  \begin{subequations}\label{eq:local.balance}
    \begin{align}
      \begin{split}\label{eq:local.balance:momentum}
        \visc(\GRAD(\rT\uvu[T]),\GRAD\vv[T])_T
        - (\vu[T],(\vel\SCAL\GRAD)\vv[T])_T
        + \reac(\vu[T],\vv[T])_T
        - (p_T,\DIV\vv[T])_T
        \\
        + \sum_{F\in\Fh[T]}(\vec{\Phi}_{TF},\vv[T])_F
        &= (\vf,\vv[T])_T,
      \end{split}
      \\ \label{eq:local.balance:mass}
      (\DT\uvu[T],q_T)_T &= 0,
    \end{align}
  \end{subequations}
  where $p_T\coloneq p_{h|T}$ and, for any interface $F\in\Fhi$ such that $F\subset\partial T_1\cap\partial T_2$ for distinct mesh elements $T_1,T_2\in\Th$, the numerical traces of the flux are continuous in the sense that
  \begin{equation}\label{eq:flux.continuity}
    \vec{\Phi}_{T_1F} + \vec{\Phi}_{T_2F}=\vec{0}.
  \end{equation}
\end{theorem}
\begin{proof}
  (i) \emph{Local momentum balance.}
  Let $\uvv\in\UhD$ be fixed.
  Expanding $\mathrm{a}_{\visc,h}$ according to its definition \eqref{eq:ah.visc} then using, for all $T\in\Th$, the definition \eqref{eq:rT} of $\rT\uvv[T]$ with $\cvw=\rT\uvu[T]$ and the definition \eqref{eq:RpT} of the boundary residual operator with $\uvw[T]=\uvu[T]$ and $\underline{\vec{\alpha}}_{\partial T}=\underline{\vec{\Delta}}_{\partial T}^k\uvv[T]$, we can write
  \begin{equation*} 
    \mathrm{a}_{\visc,h}(\uvu,\uvv)
    = \sum_{T\in\Th}\left(
    \visc(\GRAD(\rT\uvu[T]),\GRAD\vv[T])_T
    - \sum_{F\in\Fh[T]}(-\visc\GRAD(\rT\uvu[T]) + \vec{R}_{TF}^k\uvu[T],\vv[F]-\vv[T])_F
    \right),
  \end{equation*}
  where the viscous stabilization was reformulated using \eqref{eq:sT.visc:bis} then \eqref{eq:RpT}.
  In a similar way, expanding $\mathrm{a}_{\vel,\reac,h}$ then, for all
  $T\in\Th$, $\Gb\uvv[T]$ according to their respective definitions \eqref{eq:ah.vel.reac} and \eqref{eq:Gb}, we have that
  \begin{equation*} 
    \mathrm{a}_{\vel,\reac,h}(\uvu,\uvv)
    = \sum_{T\in\Th}\left(
    -(\vu[T],(\vel\SCAL\GRAD)\vv[T])_T + \reac(\vu[T],\vv[T])_T
    - \sum_{F\in\Fh[T]}(\svel\vu[T] + \svel^-(\vu[T]-\vu[F]),\vv[F]-\vv[T])_F
    \right).
  \end{equation*}
  Finally, recalling the definition \eqref{eq:bh} of $\mathrm{b}_h$ and \eqref{eq:DT} of the discrete divergence operator, we have that
  \begin{equation*} 
    \mathrm{b}_h(\uvv,p_h)
    = \sum_{T\in\Th}\left(
    -(p_h,\DIV\vv[T])_T - \sum_{F\in\Fh[T]}(p_T\normal_{TF},\vv[F]-\vv[T])_F
    \right).
  \end{equation*}
  Plugging the above expressions into \eqref{eq:discrete:momentum}, we conclude that
  \begin{multline*} 
    \sum_{T\in\Th}\Bigg(
    \visc(\GRAD(\rT\uvu[T]),\GRAD\vv[T])_T
    - (\vu[T],(\vel\SCAL\GRAD)\vv[T])_T
    + \reac(\vu[T],\vv[T])_T
    - (p_T,\DIV\vv[T])_T
    \\
    - \sum_{F\in\Fh[T]}(\vec{\Phi}_{TF},\vv[F]-\vv[T])_F
    \Bigg) = (\vf,\vv[h]).
  \end{multline*}
  Selecting now $\uvv$ such that $\vv[T]$ spans $\Poly{k}(T)^d$ for a selected mesh element $T\in\Th$ while $\vv[T'] = \vec{0}$ for all $T'\in\Th\setminus\{T\}$ and $\vv[F] = \vec{0}$ for all $F\in\Fh$, we obtain the local momentum balance \eqref{eq:local.balance:momentum}.
  On the other hand, selecting $\uvv$ such that $\vv[T] = \vec{0}$ for all $T\in\Th$, $\vv[F]$ spans $\Poly{k}(F)^d$ for a selected interface $F\in\Fhi$ such that $F\subset\partial T_1\cap\partial T_2$ for distinct mesh elements $T_1,T_2\in\Th$, and $\vv[F'] = \vec{0}$ for all $F'\in\Fh\setminus\{F\}$ yields the flux continuity \eqref{eq:flux.continuity} after observing that $\left(\vec{\Phi}_{T_1F}+\vec{\Phi}_{T_2F}\right)\in\Poly{k}(F)^d$.
  \medskip\\
  (ii) \emph{Local mass balance.}
  We start by observing that \eqref{eq:discrete:mass} holds in fact for all $q_h\in\Poly{k}(\Th)$, not necessary with zero mean value on $\Omega$.
  This can be easily checked using the definition \eqref{eq:bh} of $\mathrm{b}_h$ and \eqref{eq:DT} of the discrete divergence to write
  $$
  b_h(\uvu,1)
  = -\sum_{T\in\Th}(\DT\uvu[T],1)_T
  = -\sum_{T\in\Th}\sum_{F\in\Fh[T]}(\vu[F]\SCAL\normal_{TF},1)_F
  = -\sum_{F\in\Fh}\sum_{T\in\Th[F]}(\vu[F]\SCAL\normal_{TF},1)_F
  = 0,
  $$
  where we have denoted by $\Th[F]$ the set of elements that share $F$ and the conclusion follows from the single-valuedness of $\vu[F]$ for any $F\in\Fhi$ and the fact that $\vu[F]=\vec{0}$ for any $F\in\Fhb$.
  In order to prove the local mass balance \eqref{eq:local.balance:mass}, it then suffices to take $q_h$ in \eqref{eq:discrete:mass} equal to $q_T$ inside $T$ and zero elsewhere.
\end{proof}


\bibliographystyle{plain}
\bibliography{osho}

\end{document}